\documentclass{article}

\usepackage[twoside,
paperwidth=210mm,
paperheight=297mm,
textheight=622pt,
textwidth=468pt,
centering,
headheight=50pt,
headsep=12pt,
footskip=18pt,
footnotesep=24pt plus 2pt minus 12pt,
columnsep=2pc]{geometry}

\usepackage[auth-sc]{authblk}

\usepackage{mathtools}

\allowdisplaybreaks

\usepackage{amssymb}
\usepackage[mathscr]{eucal}

\usepackage{mismath}

\usepackage{hyperref}

\usepackage{bm} 

\usepackage{graphicx}
\graphicspath{{figs/}}

\usepackage{tikz}
\usetikzlibrary{math}
\usetikzlibrary{calc}

\usepackage{subcaption}

\usepackage{multirow}

\usepackage{makecell}
\setcellgapes{2pt}

\usepackage{booktabs}

\usepackage{siunitx}
\usepackage{algorithm}
\usepackage{algpseudocode}

\usepackage[skins]{tcolorbox}
\newtcolorbox{stepbox}[2][]{%
  enhanced,
  attach boxed title to top center={yshift=-3mm,yshifttext=-1mm},
  colframe=blue!75!black,
  colbacktitle=red!80!black,
  fonttitle=\bfseries,
  title=#2,#1
}

\usepackage{enumitem}   

\usepackage{amsmath, amssymb, amsthm, bm, stmaryrd, mathrsfs, amsfonts}
\SetSymbolFont{stmry}{bold}{U}{stmry}{m}{n} 
\usepackage{geometry}
\usepackage{graphicx}
\usepackage{hyperref}
\hypersetup{hidelinks}
\graphicspath{{Figures/}}
\DeclareMathOperator{\as}{as}

\usepackage[capitalise]{cleveref}

\newtheorem{theorem}{Theorem}[section]
\newtheorem{lemma}[theorem]{Lemma}
\newtheorem{remark}[theorem]{Remark}
\newtheorem{proposition}[theorem]{Proposition}

\newcommand{\jump}[1]{\llbracket #1 \rrbracket}  
\newcommand{\spint}[2]{(\,#1\,,\,#2\,)}         

\providecommand{\keywords}[1]{\textbf{\textit{Keywords---}} #1} 

\crefname{equation}{}{}

\usepackage{lineno}
\usepackage[numbers,sort&compress]{natbib}

\numberwithin{figure}{section}
\numberwithin{table}{section}

\title{A time explicit multiscale framework for heterogeneous elastic wave propagation based on multipoint stress discretization}
\author[1]{Xiang Zhong}
\author[1]{Eric T. Chung}
\author[2]{Shubin Fu\thanks{Corresponding author. (Email address: \href{mailto:sfu@eitech.edu.cn}{sfu@eitech.edu.cn})}}
\affil[1]{Department of Mathematics, The Chinese University of Hong Kong, Shatin, Hong~Kong~SAR, China.}
\affil[2]{Eastern Institute of Technology, Ningbo, Zhejiang, 315200, China.}
\date{}

\begin{document}
\maketitle
\begin{abstract}
  Simulating elastic wave propagation in heterogeneous media presents significant mathematical and computational challenges, since resolving fine-scale material variations requires extremely fine spatial discretizations, while mixed stress--displacement formulations typically lead to large saddle-point systems that are not well suited for efficient time integration. Moreover, conventional multiscale reductions generally produce nontrivial coarse mass matrices, thereby requiring additional linear solves at every time step and diminishing the advantages of explicit schemes. To address these difficulties, we develop a time explicit multiscale method based on a multipoint stress control volume discretization. The central innovation is a unified spatial--temporal reduction strategy that simultaneously removes two dominant algebraic bottlenecks: the complex global fine-scale mixed solve and the repeated coarse-scale mass solve arising in conventional multiscale time stepping. More specifically, the stress and auxiliary rotation variables are eliminated locally, leading to a symmetric positive definite reduced stiffness operator, upon which a multiscale space is constructed using a density-weighted local spectral problem. The associated projected mass bilinear form preserves the physical mass inner product and gives an identity mass matrix under a density-orthonormal auxiliary basis, leading naturally to an explicit central-difference scheme. We establish discrete energy stability under a coarse-scale CFL condition and derive convergence estimates for both the density-weighted displacement and the locally recovered stress. Numerical experiments in heterogeneous media validate the theoretical results and demonstrate the effectiveness of the proposed method for elastic wave propagation.

\end{abstract}


\section{Introduction}
Elastic wave propagation in heterogeneous media is central to seismic imaging,
nondestructive testing, composite-material design, and structural health
monitoring. Fine-scale variations of the density and elastic moduli generate
scattering, mode conversion, and localized wave patterns that must be represented
by the spatial discretization. A conventional finite element or finite difference method must therefore resolve the multiple spatial scales introduced by material heterogeneity and the range of wavelengths present in the wave field, while the associated temporal frequencies impose corresponding restrictions on the time-step size. In three dimensions, meeting these simultaneous
resolution requirements produces very large
algebraic systems; for explicit time integration it also imposes a restrictive
fine-grid Courant--Friedrichs--Lewy (CFL) condition. These difficulties become
especially severe for heterogeneous media and long-time simulations.

Mixed stress--displacement formulations constitute a classical framework for
linear elasticity, providing direct stress approximations and robustness in the
nearly incompressible regime
\cite{BrezziFortin1991,BoffiBrezziFortin2013,Falk2008,AdamsCockburn2005,
ArnoldWinther2002,ArnoldAwanouWinther2008,ArnoldFalkWinther2007,
CockburnGopalakrishnanGuzman2010,GopalakrishnanGuzman2012,ZhongQiu2023}. Strongly symmetric
formulations constrain the discrete stress pointwise to an $H(\divg)$-conforming
space of symmetric tensors, thereby preserving stress symmetry exactly;
representative constructions include the Arnold--Winther elements and their
three-dimensional counterparts
\cite{ArnoldWinther2002,AdamsCockburn2005,ArnoldAwanouWinther2008,ZhongQiu2023}. Weakly
symmetric formulations instead approximate the stress in full matrix-valued
$H(\divg)$ spaces and enforce symmetry variationally through an auxiliary
rotation acting as a Lagrange multiplier, allowing simpler stable elements
\cite{ArnoldFalkWinther2007,CockburnGopalakrishnanGuzman2010,
GopalakrishnanGuzman2012}. Although mixed formulations provide accurate stresses,
standard discretizations generally lead to large saddle-point systems.
Multipoint stress mixed finite element methods alleviate this cost by localizing
the stress mass matrix, permitting the local elimination of stress and rotation
and yielding a cell-centered displacement system
\cite{AmbartsumyanEtAl2020,AmbartsumyanEtAl2021}. More recently, Fu and Zhao
introduced a multipoint stress control volume method using piecewise constant
approximations and no special quadrature rule \cite{Fu2025}. With a suitable
rotation space, both stress and rotation can be eliminated locally, yielding a
symmetric positive definite displacement system while retaining local
conservation and an inexpensive stress recovery. This structure is well suited
to time-dependent elasticity, but a fine-grid solve remains expensive when the
medium contains many unresolved scales.

Multiscale methods address the spatial complexity by incorporating fine-scale
coefficient information into coarse basis functions. 
Existing multiscale model reduction methodologies include multiscale finite element methods (MsFEM) \cite{YTY2009,HouWu1997}, generalized multiscale finite element method
(GMsFEM) \cite{EfendievGalvisHou2013}, localized orthogonal decomposition method (LOD) \cite{AMDP2014}, variational multiscale finite element methods \cite{HFJL1998}, numerical upscaling \cite{PDR2016}, heterogeneous multiscale methods \cite{EW2003,EWBE2003,BEYH2005} and so on. Building on GMsFEMs and LOD, the constraint energy-minimizing GMsFEM (CEM-GMsFEM) constructs localized multiscale basis functions over oversampled regions and provides robust approximations for high-contrast media \cite{Eric2018,EricYETY2023}. These ideas have been developed for elasticity
and elastic waves in several forms, including GMsFEM for heterogeneous anisotropic
elastic-wave propagation \cite{GaoFuGibsonChungEfendiev2015}, high-order
multiscale bases for time-domain elastic waves \cite{FuGaoChung2019}, and
locking-free multiscale mixed formulations for high-contrast elasticity
\cite{chung2025locking,ChungKimZhong2026}.

To fully exploit the efficiency of the fine-scale multipoint discretization and the subsequent multiscale reduction proposed above, the computational cost of time integration must also be addressed. Spectral multiscale spaces coupled by a
discontinuous Galerkin formulation can produce block-diagonal mass matrices and
support explicit time stepping \cite{ChungEfendievLeung2014}. Localized
orthogonal decomposition has likewise been used to reduce spatial complexity and
relax the fine-scale CFL restriction \cite{MaierPeterseim2019}. Within the CEM
framework, an explicit and energy-conserving Petrov--Galerkin method was developed
for the scalar wave equation \cite{Cheung2021}. Nevertheless, a standard Galerkin reduction generally
produces a non-diagonal coarse mass matrix, whose repeated inversion diminishes
the benefit of an explicit scheme. Moreover, existing explicit multiscale wave
methods do not simultaneously exploit a locally condensed multipoint stress
formulation, preserve the physical density inner product at the reduced level,
and provide local recovery of the stress. Bridging these features is the main
objective of the present work.

We develop a time explicit multiscale method for heterogeneous elastic
waves. The starting point is a multipoint stress control volume discretization of
the mixed stress--displacement--rotation system. Local elimination of stress and
rotation gives a displacement-only fine-scale equation with a symmetric positive
definite stiffness bilinear form. We then define a $\rho$-weighted ($\rho$ is the density) auxiliary
spectral problem on each coarse block and construct localized multiscale trial functions
on oversampled regions. The density-orthogonal projection and the associated
mass bilinear form
make the coarse mass matrix the identity when the auxiliary basis is
$\rho$-orthonormal. Consequently, the reduced coefficients are advanced by a
central-difference update without solving a linear system at each time step.

The main contributions are summarized as follows.
\begin{enumerate}
  \item The main novelty is a unified spatial--temporal reduction framework that
simultaneously simplifies the mixed spatial discretization and the time
evolution. Spatially, local multipoint stress condensation and multiscale model
reduction reduce the globally evolved unknowns to the cell-centered displacement
while retaining inexpensive local stress recovery. Temporally, a
density-adapted Petrov--Galerkin projection renders the physical coarse-scale
mass matrix exactly the identity, enabling genuinely explicit time stepping
without repeated mass solves. This co-design therefore avoids both a global
mixed solve and repeated coarse-scale linear solves.
  \item We establish a discrete energy estimate under a coarse-scale CFL
  condition and derive convergence estimates for both the density-weighted
  displacement and the locally recovered stress. 
  \item Numerical experiments for binary scattering and strongly correlated
  random media demonstrate the accuracy and effectiveness of the method.
\end{enumerate}

The remainder of the paper is organized as follows. Section~\ref{sec: Preliminaries} introduces the
elastic-wave model, the multipoint stress control volume discretization, and the
reduced displacement formulation. Section~\ref{sec:multiscale} presents the density-weighted
spectral problem, the localized multiscale basis, and the explicit scheme.
Section~\ref{sec:analysis} establishes stability and convergence. Section~\ref{Numerical experiments} reports the numerical
experiments, and Section~\ref{conclusions} concludes the paper.

\section{Preliminaries}\label{sec: Preliminaries}
In this section, we first present the linear elastic problem in heterogeneous media. Then we introduce the key definitions and notation that will be employed throughout our subsequent analysis.
\subsection{Model problem}
We consider the elastic wave equation in a polyhedral domain
$\Omega\subset\mathbb{R}^d$ ($d=2,3$).
Let $\underline{\sigma}(\cdot,t):\Omega\to\mathbb{R}^{d\times d}$ be the symmetric stress tensor,
$\bm{u}(\cdot,t)$ the displacement field, and $\rho$ the mass density.
The equations read
\begin{align}
\mathcal{A}\underline{\sigma} &= \underline{\varepsilon}(\bm{u}) && \text{in } \Omega\times(0,T], \label{eq:bulk} \\
\rho\,\frac{\partial^2\bm{u}}{\partial t^2} - \divg\underline{\sigma} &= \bm{f} && \text{in } \Omega\times(0,T], \label{eq:momentum}
\end{align}
where $\bm{f}(\bm{x},t)$ is a given source term. The problem is supplemented with the homogeneous
Dirichlet boundary condition $\bm{u}=\bm{0}$ on $[0,T]\times\partial \Omega$ and initial conditions $\bm{u}(\bm{x},0)=\bm{u}_0(\bm{x})$ and $\bm{u}_t(\bm{x},0)=\bm{v}_0(\bm{x})$ in $\Omega$.
Here $\underline{\varepsilon}(\bm{u}) = \tfrac12\bigl(\nabla\bm{u}+(\nabla\bm{u})^T\bigr)$ is the linearized strain tensor,
and $\mathcal{A}$ is the inverse of the elasticity operator, defined by
\begin{equation}
\mathcal{A}\underline{\tau} := \frac{1}{2\mu}\Bigl(\underline{\tau} - \frac{\lambda}{d\lambda+2\mu}(\operatorname{tr}\underline{\tau})\,\underline{I}\Bigr),
\end{equation}
where $\lambda,\mu$ are the Lam\'e coefficients and $\underline{I}$ is the identity matrix in $\mathbb{R}^{d\times d}$. Let $E$ be Young's modulus and $\nu$ be Poisson's ratio, then the Lam$\rm \acute{e}$ coefficients $\lambda$ and $\mu$ are denoted as
\begin{equation}
	\label{Lame constants}
	\lambda\coloneqq\frac{E\nu}{(1+\nu)(1-2\nu)},\quad  \mu\coloneqq\frac{E}{2(1+\nu)}.
\end{equation}
For nearly incompressible materials, $\lambda$ is large in comparison with $\mu$ (More precisely, $\lambda\to\infty$ as Poisson’s ratio $\nu\to 1/2$ while $\mu$ is bounded). In this paper, $\lambda$ and $\mu$ are highly heterogeneous in space and possibly high contrast. We then recall some notation in \cite{Fu2025}.  Define 
\(
\bm{H}(\operatorname{div};\Omega) \coloneqq \{ \bm{v} \in L^2(\Omega), \; \nabla\cdot\bm{v} \in L^2(\Omega) \},
\) equipped with the norm
\(
\|\bm{v}\|_{H(\operatorname{div};\Omega)}^2 = \|\bm{v}\|_{L^2(\Omega)}^2 + \|\nabla\cdot\bm{v}\|_{L^2(\Omega)}^2
\)
and we use $\underline{H}(\operatorname{div};\Omega)$ to represent the tensor field where each row belongs to $\bm{H}(\operatorname{div};\Omega)$. Denote the following spaces:
\[
\underline{\Sigma}\coloneqq\{\underline{\tau} \in \underline{H}(\operatorname{div} ; \Omega)\}, \qquad 
\bm{U}\coloneqq\bm{L}^2(\Omega), \qquad 
\Gamma\coloneqq L^2(\Omega).
\]
For the mixed formulation, we introduce the Lagrange multiplier
$\gamma = \frac12\bigl(-\frac{\partial u_1}{\partial y} + \frac{\partial u_2}{\partial x}\bigr)$ for $\bm{u}=(\bm{u}_1, u_2)^T$
( when $d=2$) such that $\mathcal{A}\underline{\sigma} = \nabla\bm{u} - \gamma\,\underline{\delta}$ with the constant anti‑symmetric tensor
$\underline{\delta} = \begin{pmatrix} 0 & -1 \\ 1 & 0 \end{pmatrix}$.
The weak form of the static part then becomes:
Find $(\underline{\sigma},\bm{u},\gamma)\in\underline{\Sigma}\times\bm{U}\times\Gamma$ such that (for $d=2$)
\begin{subequations}
\begin{align}
(\, \mathcal{A}\underline{\sigma} \,,\, \underline{w} \,) + 
(\, \bm{u} \,,\, \operatorname{div}\underline{w} \,) + 
(\, \operatorname{as}(\mathcal{A}\underline{w}) \,,\, \gamma \,) &= 0 
\quad &&\forall \underline{w} \in \underline{\Sigma}, \\
(\rho\frac{\partial^2\bm{u}}{\partial t^2},\bm{v}) - 
(\, \operatorname{div}\underline{\sigma} \,,\, \bm{v} \,) &= 
(\, \bm{f} \,,\, \bm{v} \,) 
\quad &&\forall \bm{v} \in \bm{U}, \\
(\, \operatorname{as}(\mathcal{A}\underline{\sigma}) \,,\, \xi \,) &= 0 
\quad &&\forall \xi \in \Gamma,
\end{align}
\end{subequations}
where $\as(\underline{w})=w_{12}-w_{21}$. The initial data is projected onto the space $\bm{U}$ by the following: find $\bm{u}(\cdot,0),\frac{\partial\bm{u}}{\partial t}(\cdot,0)\in\bm{U}$
such that for all $\bm{w}\in\bm{U}$,
\begin{subequations}
\begin{align}
(\, \rho\bm{u}(\cdot,0) \,,\, \bm{w} \,) &= (\, \rho\bm{u}_0,\bm{w}), \\
(\rho\frac{\partial \bm{u}}{\partial t}(\cdot,0),\bm{w}) &= 
(\, \rho\bm{v}_0 \,,\, \bm{w} \,).
\end{align}
\end{subequations}

Similar definitions can be found for the case \(d=3\) in \cite{Fu2025}. Note that  we restrict ourselves to the case $d = 2$ in the construction and analysis of the proposed scheme and the extension to $d = 3$ is natural. 

\subsection{Multipoint stress control volume method}
In this section, we recall the discrete formulation from \cite{Fu2025}. 
We first introduce some notation that will be used throughout.

Let $\mathcal{T}_M$ denote a partition of $\Omega$ into quadrilateral meshes, where each element $M \in \mathcal{T}_M$ is regarded as a macro-element. The set of all edges generated by this partition is denoted by $\mathcal{F}_{pr}$. For each macro-element $M$, we select one interior point and connect it to the midpoints of the edges of $M$, thereby decomposing $M$ into four subcells. This subdivision splits each edge $e \in \mathcal{F}_{pr}$ into two equal half-edges; the collection of all such half-edges is denoted by $\mathcal{F}_{pr}^{\frac12}$ (see \cite[Figure 1]{Fu2025}). Additionally, four interior subcell edges are created inside each macro-element, and the set of all such interior edges is denoted by $\mathcal{F}_{dl}$. We define
\(
\mathcal{F}_h \coloneq \mathcal{F}_{pr}^{\frac12} \cup \mathcal{F}_{dl}.
\)
The union of all subcells is written as $\mathcal{T}_h$. 

An interaction region $D$ is formed by the four subcells sharing a common vertex (see \cite[Figure 1]{Fu2025}). The collection of all interaction regions is denoted by $\mathcal{T}_D$. For an element $D$ ( which may be a macro-element or a subcell), $h_D$ stands for its diameter, and we set \(h \coloneqq \max_{D} \{h_D\}\).
For an edge $e$, $h_e$ denotes its length. For simplicity, we simply denote the fine-scale mesh size by $h$. On a boundary facet $e$, $\bm{n}_e$ is the unit outward normal vector. For an interior facet we fix $\bm{n}_e$ as one of the two possible unit normals. When no confusion arises, we simply write $\bm{n}$.

Let $k \geq 0$ be the polynomial degree. We denote by $P_k$ the space of polynomials of total degree at most $k$, and by $Q_k$ the space of polynomials of degree at most $k$ in each variable. For scalar-, vector-, and matrix- valued functions $q$, $\bm{v}$, and $\underline{\omega}$, respectively, and for two adjacent elements $E^+$ and $E^-$ sharing a common facet $e = \partial E^+ \cap \partial E^-$, the jumps are defined as
\[
\llbracket q \rrbracket \coloneqq q|_{E^+} - q|_{E^-}, \qquad
\llbracket \bm{v} \rrbracket \coloneqq\bm{v}|_{E^+} - \bm{v}|_{E^-}, \qquad
\llbracket \underline{\omega} \rrbracket \coloneq \underline{\omega}|_{E^+} - \underline{\omega}|_{E^-}.
\]
On a boundary facet we set $\llbracket q \rrbracket \coloneqq q$, $\llbracket \bm{v} \rrbracket \coloneqq\bm{v}$, and $\llbracket \underline{\omega} \rrbracket \coloneqq\underline{\omega}$.

We recall Method~1 from \cite[Section 3.1]{Fu2025}, which will be employed in our multiscale scheme. Discrete spaces are defined as follows:
\begin{equation*}\label{eq:spaces}
\begin{aligned}
\bm{\Sigma}_h^* &\coloneqq\Bigl\{\bm{\omega}_h \in \bm{P}_0(E),\; \forall E \in \mathcal{T}_h,\; 
\llbracket\bm{\omega}_h \cdot \bm{n}\rrbracket \big|_e =0,\; \forall e \in \mathcal{F}_{pr}^{\frac12}\backslash{\partial\Omega}\Bigr\}, \\[2mm]
\bm{U}_h &\coloneqq\bigl[U_h\bigr]^2 =\Bigl\{\bm{v}_h \in P_0(M),\; \forall M \in \mathcal{T}_M\Bigr\}, \\[2mm]
\Gamma_h &\coloneqq\Bigl\{\mu_h \in P_0(D),\; \forall D \in \mathcal{T}_D\Bigr\}
\end{aligned}
\end{equation*}
and $\underline{\Sigma}_h\coloneqq[\bm{\Sigma}_h^*]^2$. The discrete formulation for the model problem reads as follows:  
Find $(\sigma_h, u_h, \gamma_h) \in \underline{\Sigma}_h \times \bm{U}_h \times \Gamma_h$ such that  
\begin{subequations}\label{eq:discrete}
\begin{align}
    \bigl( \mathcal{A}\underline{\sigma}_h,\ \underline{w}_h \bigr)
    -\sum_{e\in\mathcal{F}_{dl}} \bigl( \bm{u}_h,\ \llbracket \underline{w}_h \bm{n} \rrbracket \bigr)_e
    +\bigl( \operatorname{as}(\mathcal{A}\underline{w}_h),\ \gamma_h \bigr) &= 0 
    &&\forall \underline{w}_h \in \underline{\Sigma}_h, \label{eq:discrete1} \\[2mm]
    \spint{\rho\frac{\partial^2\bm{u}_h}{\partial t^2}}{\bm{v}_h}-\sum_{e\in\mathcal{F}_{pr}^{\frac12}} \bigl( \underline{\sigma}_h \bm{n},\ \llbracket \bm{v}_h \rrbracket \bigr)_e
    &= \bigl( \bm{f},\ \bm{v}_h \bigr) 
    &&\forall \bm{v}_h \in \bm{U}_h, \label{eq:discrete2} \\[2mm]
    \bigl( \operatorname{as}(\mathcal{A}\underline{\sigma}_h),\ \xi_h \bigr) &= 0 
    &&\forall \xi_h \in \Gamma_h. \label{eq:discrete3}
\end{align}
\end{subequations}
 The initial data is projected onto the space $\bm{U}_h$ by the following: find $\bm{u}_h(\cdot,0),\frac{\partial\bm{u}_h}{\partial t}(\cdot,0)\in\bm{U}_h$
such that for all $\bm{w}_h\in\bm{U}_h$,
\begin{subequations}\label{discrete_initial_data}
\begin{align}
(\, \rho\bm{u}_h(\cdot,0) \,,\, \bm{w}_h \,) &= (\, \rho\bm{u}_0,\bm{w}_h), \label{discrete_initial_data_a}\\
(\rho\frac{\partial \bm{u}_h}{\partial t}(\cdot,0),\bm{w}_h) &= 
(\, \rho\bm{v}_0 \,,\, \bm{w}_h \,).\label{discrete_initial_data_b}
\end{align}
\end{subequations}
Note that the first term on the left-hand side of \eqref{eq:discrete1} is block-diagonal; consequently, we can eliminate $\sigma_h$ locally to obtain a reduced system. According to \cite[Section 4.4]{Fu2025}, the rotation can be further eliminated. 
Let $\bm{\sigma},\mathbf{u},\bm{\gamma}$ be the coefficient vectors representing
$\underline{\sigma}_h,\bm{u}_h,\gamma_h$ in the chosen bases.
The block structure of \eqref{eq:discrete} gives the matrix form
\begin{equation}
\mathbf{M}_u\frac{\partial^2\mathbf{u}}{\partial t^2} +
\begin{pmatrix}
\mathbf{0} & \mathbf{A}_{\sigma u} & \mathbf{0} \\
\mathbf{A}_{\sigma u}^T & \mathbf{A}_{\sigma\sigma} & \mathbf{A}_{\sigma\gamma}^T \\
\mathbf{0} & \mathbf{A}_{\sigma\gamma} & \mathbf{0}
\end{pmatrix}
\begin{pmatrix} \mathbf{u} \\ \bm{\sigma} \\ \bm{\gamma} \end{pmatrix}
= \begin{pmatrix} \mathbf{F} \\ \mathbf{0} \\ \mathbf{0} \end{pmatrix},
\end{equation}
with \((\mathbf{M}_u)_{ij}= \spint{\rho\bm{v}_j}{\bm{v}_i},(\mathbf{A}_{\sigma u})_{ij} = \sum_{e\in\mathcal{F}_{dl}} \bigl(\bm{v}_j,\,\jump{\underline{w}_i\bm{n}}\bigr)_e
= -\sum_{e\in\mathcal{F}_{\mathrm{pr}}^{1/2}} \bigl(\jump{\bm{v}_j},\,\underline{w}_i\bm{n}\bigr)_e,(\mathbf{A}_{\sigma\gamma})_{ij} = \spint{\as(\underline{w}_j)}{\xi_i}.\)

Eliminating $\bm{\sigma}$ and $\bm{\gamma}$ from the first and third block equations yields
$\bm{\sigma} = \mathbf{K}_1\mathbf{u}$ with
\begin{equation}
\mathbf{K}_1 = -\mathbf{A}_{\sigma\sigma}^{-1}
\Bigl[\,\mathbf{I} - \mathbf{A}_{\sigma\gamma}^T
\bigl(\mathbf{A}_{\sigma\gamma}\mathbf{A}_{\sigma\sigma}^{-1}\mathbf{A}_{\sigma\gamma}^T\bigr)^{-1}
\mathbf{A}_{\sigma\gamma}\mathbf{A}_{\sigma\sigma}^{-1}\Bigr]\mathbf{A}_{\sigma u}^T .
\end{equation}
Substituting into the second equation we obtain a second‑order ODE system for the displacement alone:
\begin{equation}\label{eq:schur_system}
\mathbf{M}_u\frac{\partial^2\mathbf{u}}{\partial t^2} + \mathbf{K}\mathbf{u} = \mathbf{F},
\end{equation}
where $\mathbf{K} = -\mathbf{A}_{\sigma u}\mathbf{K}_1$ is the symmetric positive definite stiffness
matrix, as follows from the local solvability of the multipoint stress
formulation and the coercivity of the corresponding reduced energy form.

\subsection{Variational form for reduced cell-centered system}\label{reduced cell-centered system}
For the sake of analysis, we derive the corresponding variational form of \eqref{eq:schur_system} in this section.

It is straightforward to verify that for any piecewise‑constant stress $\underline{\tau}_h$,
\begin{equation}\label{eq:divergence_identity}
\sum_{e\in\mathcal{F}_{pr}^{\frac12}}\bigl(\underline{\tau}_h \bm{n},\; \llbracket \bm{v}_h \rrbracket\bigr)_e
= \sum_{e\in\mathcal{F}_{dl}}\bigl(\bm{v}_h,\; \llbracket \underline{\tau}_h \bm{n} \rrbracket\bigr)_e,
\end{equation}
holds
for all $\bm{v}_h\in\bm{U}_h$. 
Combining this equivalence, we define a jump operator $\bm{\mathcal{J}}\colon \bm{U}_h\to\underline{\Sigma}_h$ such that for all $\bm{v}_h\in\bm{U}_h$,
\begin{equation}\label{eq:jump_operator}
\begin{aligned}
(\bm{\mathcal{J}}(\bm{v}_h),\, \underline{w}_h)
\coloneqq\sum_{e\in\mathcal{F}_{dl}} \bigl(\bm{v}_h,\; \llbracket \underline{w}_h \bm{n} \rrbracket\bigr)_e= -\sum_{e\in\mathcal{F}_{pr}^{\frac12}} \bigl(\llbracket \bm{v}_h \rrbracket,\; \underline{w}_h \bm{n}\bigr)_e,
\qquad \forall \underline{w}_h \in \underline{\Sigma}_h.
\end{aligned}
\end{equation}
Let $\operatorname{as}^*: \Gamma_h \rightarrow \underline{\Sigma}_h$ be the adjoint of $\operatorname{as}$ with respect to the $L^2$ inner product, i.e., 
\[
\bigl(\operatorname{as}(\underline{w}_h),\; \gamma_h\bigr) = \bigl(\underline{w}_h,\; \operatorname{as}^*(\gamma_h)\bigr), \quad \forall \underline{w}_h \in \underline{\Sigma}_h,\; \gamma_h \in \Gamma_h.
\]
Then \cref{eq:discrete1} yields:
\begin{equation}\label{eq:sigma_expression}
\underline{\sigma}_h = \mathcal{A}^{-1}\bigl(\bm{\mathcal{J}}(\bm{u}_h) - \mathcal{A}\operatorname{as}^*(\gamma_h)\bigr).
\end{equation}
Substituting \cref{eq:sigma_expression} into \cref{eq:discrete3} gives
$
\bigl(\operatorname{as}\bigl(\bm{\mathcal{J}}(\bm{u}_h) - \mathcal{A}(\operatorname{as}^*(\gamma_h))\bigr),\; \xi_h\bigr) = 0,$ for all $\xi_h \in \Gamma_h.$
This uniquely determines $\gamma_h$ as a function of $u_h$, denoted by $\gamma_h(\bm{u}_h) $, satisfying the orthogonality condition
\begin{equation}\label{eq:orthogonality}
\left(\bm{\mathcal{J}}(\bm{u}_h) - \mathcal{A}\left(\operatorname{as}^*(\gamma_h(\bm{u}_h))\right),\; \operatorname{as}^*(\xi_h)\right) = 0, 
\quad \forall \xi_h \in \Gamma_h.
\end{equation}
We define the weakly symmetric stress projection
$\bm{\mathcal{P}}_{\star}: \bm{U}_h \to \underline{\Sigma}h$ by
\begin{equation}
\bm{\mathcal{P}}_{\star}(\bm{u}_h) := \bm{\mathcal{J}}(\bm{u}_h) - \mathcal{A}\bigl(\operatorname{as}^*(\gamma_h(\bm{u}_h))\bigr). \label{eq:Pstar-wave}
\end{equation}
By~\eqref{eq:orthogonality}, $\bm{\mathcal{P}}_{\star}$ is the orthogonal projection (with respect to
$(\mathcal{A}^{-1}\,\cdot\,,\cdot)$) of $\bm{\mathcal{J}}(\bm{u}_h)$ onto the subspace of weakly symmetric stresses.
Using \eqref{eq:jump_operator}, \eqref{eq:discrete2} becomes
\begin{equation}
\spint{\rho\frac{\partial^2\bm{u}_h}{\partial t^2}}{\bm{v}_h}
+ \bigl(\bm{\mathcal{J}}(\bm{v}_h), \underline{\sigma}_h\bigr)
= \spint{\bm{f}}{\bm{v}_h}, \quad \forall\bm{v}_h \in \bm{U}_h.
\end{equation}
Substituting $\underline{\sigma}_h =\mathcal{A}^{-1}\bm{\mathcal{P}}_{\star}(\bm{u}_h)$ and using the
orthogonality~\eqref{eq:orthogonality}, we obtain the symmetric formulation:
\begin{equation}
\spint{\rho\frac{\partial^2\bm{u}_h}{\partial t^2}}{\bm{v}_h}
+ \bigl(\bm{\mathcal{J}}(\bm{v}_h) - \mathcal{A}\operatorname{as}^*(\gamma_h(\bm{v}_h)),\,
\mathcal{A}^{-1}\bigl(\bm{\mathcal{J}}(\bm{u}_h) - \mathcal{A}\operatorname{as}^*(\gamma_h(\bm{u}_h))\bigr)\bigr)
= \spint{\bm{f}}{\bm{v}_h}, \quad \forall\bm{v}_h \in \bm{U}_h.
\end{equation}
Thus, the energy bilinear form is
\begin{equation}
a(\bm{u}_h, \bm{v}_h)
:= \bigl(\mathcal{A}^{-1} \bm{\mathcal{P}}_{\star}(\bm{u}_h),\,
\bm{\mathcal{P}}_{\star}(\bm{v}_h)\bigr)
= \bigl(\mathcal{A}^{-1}\bigl(\bm{\mathcal{J}}(\bm{u}_h) - \mathcal{A}\operatorname{as}^*(\gamma_h(\bm{u}_h))\bigr),\,
\bm{\mathcal{J}}(\bm{v}_h) - \mathcal{A}\operatorname{as}^*(\gamma_h(\bm{v}_h))\bigr). \label{eq:ah-wave}
\end{equation}
We also define the energy norm $\norm{\bm{v}}_a^2=a(\bm{v},\bm{v})$. The reduced weak form of the semi-discrete elastic wave equation therefore reads:
\begin{equation}
\spint{\rho\frac{\partial^2\bm{u}_h}{\partial t^2}}{\bm{v}_h}
+ a(\bm{u}_h, \bm{v}_h)
= \spint{\bm{f}}{\bm{v}_h}, \qquad \forall\bm{v}_h \in \bm{U}_h\; t \in (0,T], \label{eq:reduced-wave}
\end{equation}
with the initial data satisfying (\ref{discrete_initial_data}).
\begin{remark}
The bilinear form $a(\cdot,\cdot)$ is symmetric, positive definite, and corresponds to the Schur-complement energy obtained after eliminating the stress and rotation variables from the
original mixed formulation~\eqref{eq:discrete}.  Equation~\eqref{eq:reduced-wave} operates
solely on the displacement degrees of freedom in $\bm{U}_h$.
\end{remark}

\section{The multiscale method}\label{sec:multiscale}
In this section, we will present the construction of our multiscale method.
We first define some notations. 
Let $\mathcal{T}_H\coloneqq\cup_{i=1}^N\{K_i\}$ denote a conforming quasi-uniform partition of the domain $\Omega$ into quadrilateral elements, where $H$ represents the coarse mesh size and $N$ is the total number of coarse elements. We refer to $\mathcal{T}_H$ as the coarse grid, where each coarse element $K_i$ is further subdivided into a connected union of fine-grid blocks in $\mathcal{T}_M$. In this setting, $\mathcal{T}_M$ serves as the fine-scale grid in our multiscale framework, refining the coarse grid $\mathcal{T}_H$. Define the oversampled
region $K_{i,m}$ by extending the coarse element $K_i$ by $m$ coarse grid layers.

\subsection{Multiscale basis construction}
We first employ the GMsFEM spectral problem to construct the multiscale test basis functions on each generic coarse element $K$. Then, by means of constrained energy minimization, we construct the multiscale trial basis functions. Subsequently, we derive the coarse-scale model using a Petrov–Galerkin projection and a symmetric formulation, which leads to an explicit time-marching coarse-scale scheme.

For each coarse element \( K_i \) (\( 1 \leq i \leq N \)), let $\bm{U}_h(K_i)$ denote the restriction of $U_h$ to $K_i$. We solve the following spectral problem: find a real number $\lambda_j^i$ and a function $\bm{\phi}_j^i \in \bm{U}_h(K_i)$ such that
\begin{equation}\label{eq:spectral_problem}
    a_i(\bm{\phi}_j^i,\bm{w})
    =
    \lambda_j^i/H^2 (\rho \bm{\phi}_j^i,\bm{w}),
    \qquad
    \forall \bm{w}\in \bm{U}_h(K_i),
\end{equation}
where
the local energy form is \(a_i(\bm v,\bm w):=\bigl(\mathcal A^{-1}\bm{\mathcal P}_\star^i(\bm v),\bm{\mathcal P}_\star^i(\bm w)\bigr)_{K_i}.\) The local operators \(\bm{\mathcal J}_i\) and \(\operatorname{as}_i^*\) are the restrictions of \(\bm{\mathcal J}\) and \(\operatorname{as}^*\), respectively, to the coarse block \(K_i\). More precisely,
 $\bm{\mathcal{J}}_i: \bm{U}_h(K_i)\rightarrow\underline{\Sigma}_h(K_i)$ defined by
  \begin{equation*}
    (\bm{\mathcal{J}}_i(\bm{u}_h), w_h)_i \coloneqq \sum_{e \in \mathcal{F}_{dl} \cap K_i} 
    \bigl(\bm{u}_h,\; \llbracket w_h n \rrbracket\bigr)_e, \qquad \forall w_h \in \underline{\Sigma}_h(K_i),
  \end{equation*}
  where $\underline{\Sigma}_h(K_i)$ is a conforming subspace of $\underline{\Sigma}_h$. And
   $\operatorname{as}_i^*: \Gamma_h(K_i) \rightarrow \underline{\Sigma}_h(K_i)$ defined as the adjoint of $\operatorname{as}_i$: for all $\underline{\xi}_h\in\Gamma_h(K_i)$,
  \begin{equation*} \bigl(\operatorname{as}_i(\underline{\tau}_h),\; \underline{\xi}_h\bigr)_{K_i} = 
    \bigl(\underline{\tau}_h,\; \operatorname{as}_i^*(\underline{\xi}_h)\bigr)_{K_i},\quad\forall \underline{\tau}_h\in\underline{\Sigma}_h(K_i).
  \end{equation*}
For every \(\bm v\in\bm U_h(K_i)\), define \(\gamma_h^i(\bm v)\in\Gamma_h(K_i)\) as the unique function satisfying the local orthogonality condition\begin{equation}\label{eq:local-orthogonality}\bigl(\bm{\mathcal J}_i(\bm v)-\mathcal A\operatorname{as}_i^*(\gamma_h^i(\bm v)),\operatorname{as}_i^*(\xi_h)\bigr)_{K_i}=0,\qquad\forall \xi_h\in\Gamma_h(K_i).\end{equation}The local weakly symmetric stress projection is then defined by\begin{equation}\label{eq:local-Pstar}\bm{\mathcal P}_\star^i(\bm v):=\bm{\mathcal J}_i(\bm v)-\mathcal A\operatorname{as}_i^*(\gamma_h^i(\bm v)).\end{equation}
We assume the normalization $(\rho \bm{\phi}_j^i,\bm{\phi}_{j'}^i)=\delta_{jj'}$. Define the $L^2$-weighted norm $\norm{\cdot}_{\rho(K_i)}$ as $\norm{\bm{v}}_{\rho(K_i)}^2=(\rho \bm{v},\bm{v})_{K_i}$ and the energy norm $\norm{\bm{v}}_{a(K_i)}^2=a_i(\bm{v},\bm{v})$. The global version is $\norm{\bm{v}}_\rho^2=(\rho \bm{v},\bm{v}), \norm{\bm{v}}_a^2=a(\bm{v},\bm{v})$, respectively. Note that the above eigenvalue problem (\ref{eq:spectral_problem}) is solved on the fine mesh in the actual computations.

Let $\lambda_j^i$ be the eigenvalues of \cref{eq:spectral_problem} arranged in ascending order. We will use the first $l_i$ eigenfunctions to construct our local multiscale test space $\bm{V}_{\mathrm{aux}}^i \coloneqq \operatorname{span}\{\bm{\phi}_j^i \mid 1 \leq j \leq l_i\},
\)
with $l_i \leq \dim \bm{U}_h(K_i)$. The global multiscale test space $\bm{V}_{\mathrm{aux}}$ is the direct sum of these local spaces, namely
$
\bm{V}_{\mathrm{aux}} = \bigoplus_{i=1}^{N} \bm{V}_{\mathrm{aux}}^i.
$
We define the local projection operator $\pi_i$ as
\(
    \pi_i\colon \bm{U}_h(K_i)\to \bm{V}_{\rm aux}^i,
\)
which is given by
\(\pi_i(\bm{v})=\sum_{j=1}^{l_i}
    (\rho \bm{v},\bm{\phi}_j^i)\,\bm{\phi}_j^i.
\)
Accordingly, the global version is defined as $\pi\coloneq \sum_{i=1}^N\pi_i$. Note that the domain of $\pi$ can be extended to $\bm{L}^2(K_i)$.
For each basis function $\bm{\phi}_j^i$, the global multiscale basis
function is defined by
\begin{equation}\label{global_mini}
    \bm{\psi}_j^i
    =
    \operatorname*{arg\,min}_{\bm{\psi}\in \bm{U}_h}
    \left\{
        a(\bm{\psi},\bm{\psi}):
        \pi(\bm{\psi})=\bm{\phi}_j^i
    \right\}.
\end{equation}
By introducing a Lagrange multiplier, the minimization problem (\ref{global_mini}) is equivalent to the following variational problem: find \( \bm{\psi}_j^{i} \in \bm{U}_h \) and \( \bm{\xi}_j^{i} \in \bm{V}_{\mathrm{aux}} \) such that
\[
\begin{aligned}
a \left( \bm{\psi}_j^{i}, \bm{\psi} \right) + \left(\rho \bm{\psi}, \bm{\xi}_j^{i} \right) &= 0 \quad \text{for all } \bm{\psi} \in \bm{U}_h, \\
\left(\rho(\bm{\psi}_j^{i} - \bm{\phi}_j^{i}), \bm{\mu} \right) &= 0 \quad \text{for all } \bm{\mu} \in \bm{V}_{\mathrm{aux}}.
\end{aligned}
\]
For the oversampling region \(K_{i,m}\), define the restricted auxiliary space \(\bm V_{\rm aux}(K_{i,m}):=\bigoplus_{K_\ell\subset K_{i,m}}\bm V_{\rm aux}^\ell.\) The localized multiscale basis $\bm{\psi}_{j,m}^i$ is defined similarly on $K_{i,m}$:
\begin{equation}\label{local_mini}
    \bm{\psi}_{j,m}^i
    =
    \operatorname*{arg\,min}_{\bm{\psi}\in \bm{U}_h(K_{i,m})}
    \left\{
        a(\bm{\psi},\bm{\psi}):
        \pi(\bm{\psi})=\bm{\phi}_j^i
    \right\},
\end{equation}
and it is extended by zero outside $K_{i,m}$.
Introducing a Lagrange multiplier, we obtain the following equivalent variational formulation of the minimization problem (\ref{local_mini}): find \( \bm{\psi}_{j,m}^{i} \in \bm{U}_h(K_{i,m}) \) and \( \bm{\xi}_{j,m}^{i} \in \bm{V}_{\mathrm{aux}}(K_{i,m}) \) such that
\[
\begin{aligned}
a \left( \bm{\psi}_{j,m}^{i}, \bm{\psi} \right) + \left(\rho \bm{\psi}, \bm{\xi}_{j,m}^{i} \right) &= 0 \quad \text{for all } \bm{\psi} \in \bm{U}_h(K_{i,m}), \\
\left( \rho(\bm{\psi}_{j,m}^{i} - \bm{\phi}_j^{i}), \bm{\mu} \right) &= 0 \quad \text{for all } \bm{\mu} \in \bm{V}_{\mathrm{aux}}(K_{i,m}).
\end{aligned}
\]
Then the multiscale trial space is defined as
\( \bm{V}_{\rm ms}=\operatorname{span}\{\bm{\psi}_{j,m}^i\}
\) and the global multiscale trial space is $\bm{V}_{\rm glo}=\operatorname{span}\{\bm{\psi}_{j}^i\}$. In particular, we define \( \widetilde{\bm{V}} \) as the null space of the projection \( \pi \), namely, \( \widetilde{\bm{V}} = \{ \bm{v} \in \bm{U}_h : \pi(\bm{v}) = \bm{0} \} \). Then for any \( \bm{\psi}_j^{i} \in \bm{V}_{\mathrm{glo}} \), we have
\(
a(\bm{\psi}_j^i, \bm{v}) = 0,\) for all \(\bm{v} \in \widetilde{\bm{V}}.
\)
Thus, \(\widetilde{\bm{V}} \subset \bm{V}_{\mathrm{glo}}^\perp\). Consequently \(\bm{V}_{\mathrm{glo}} \subset \widetilde{\bm{V}}^\perp\). Notice that the restriction map \(\pi : \widetilde{\bm{V}}^\perp \to \bm{V}_{\mathrm{aux}}\) is injective. Therefore, we have \(\dim\{\widetilde{\bm{V}}^\perp\} \leq \dim\{\bm{V}_{\mathrm{aux}}\} < \infty\). By the well-posedness of the construction of the global basis functions, we have \(\dim\{\bm{V}_{\mathrm{glo}}\} = \dim\{\bm{V}_{\mathrm{aux}}\} \geq \dim\{\widetilde{\bm{V}}^\perp\}\) and thus \(\widetilde{\bm{V}}^\perp = \bm{V}_{\mathrm{glo}}\). Hence, we have \(\widetilde{\bm{V}} = \bm{V}_{\mathrm{glo}}^\perp\). Thus, we have \(\bm{U}_h = \bm{V}_{\mathrm{glo}} \oplus \widetilde{\bm{V}}\).
\begin{proposition}[Properties of the operator $\pi$ \cite{EricYETY2023,Eric2018,ChungKimZhong2026,chung2025locking}] \label{right_inverse}
In each $K_i\in\mathcal{T}_H$, for all $\bm{v}\in \bm{U}_h(K_i)$,
\begin{equation} \label{eq:spectral-estimate}
\norm{(I-\pi)\bm{v}}_{\rho(K_i)}^2\leq H^2/\lambda_{l_i+1}^i\norm{\bm{v}}_{a(K_i)}^2\leq H^2/\Lambda\norm{\bm{v}}_{a(K_i)}^2.
\end{equation}
where $\Lambda\coloneq \min_{1\leq i\leq N} \lambda_{l_i+1}^i$, and  
\begin{equation*}\label{pi2}
\norm{\pi_i\bm{v}}_{\rho(K_i)}^2=\norm{\bm{v}}_{\rho(K_i)}^2-\norm{(I-\pi_i)\bm{v}}_{\rho(K_i)}^2\leq	\norm{\bm{v}}_{\rho(K_i)}^2.
\end{equation*}
The projection $\pi$ admits a stable right inverse. Namely, for any
$\bm{v}\in \bm{V}_{\rm aux}$, there exists $R\bm{v}\in \bm{U}_h$ such that
\[
    \pi(R\bm{v})=\bm{v},
    \quad
    \|R\bm{v}\|_a^2\le C_{\rm lift}H^{-2}\|\bm{v}\|_\rho^2,\quad {\rm supp}(R\bm{v})\subset {\rm supp}(\bm{v}),
\]
where $C_{\rm lift}>0$ is dependent on the regularity of the mesh and the eigenvalue $\lambda_{\max} = \max_{1 \leq i \leq N} \max_{1 \leq j \leq l_i} \lambda_j^{(i)}$, but independent of $h$, $H$, and the contrast.
\end{proposition}
\begin{proof}
This follows the standard analyses in \cite{EricYETY2023,Eric2018,ChungKimZhong2026,chung2025locking}, so we just omit it here.
\end{proof}
\begin{proposition}[Exponential decay of multiscale basis functions \cite{EricYETY2023,Eric2018,ChungKimZhong2026,chung2025locking}]
\label{prop:cem-decay}
Let $\bm{\psi}_j^i,\bm{\psi}_{j,m}^i$ ($m\geq 2$) be multiscale basis function obtained from (\ref{global_mini}) and (\ref{local_mini}), respectively. Then there exist constants
$C>0$ and $0<\theta<1$, independent of $h$, such that
\begin{equation}\label{decay_inequality}
    \|\bm{\psi}_j^i-\bm{\psi}_{j,m}^i\|_a
    \le
    C\theta^m
    \|\bm{\psi}_j^i\|_{a}.
\end{equation}
The constants depend only on the shape regularity of the coarse mesh and on
the stability constants of $a$, and the decay rate improves as the
spectral gap $\Lambda$ increases.
\end{proposition}
\begin{proof}
The proof follows from the stable decomposition and localization estimates in \cite{EricYETY2023,Eric2018,ChungKimZhong2026,chung2025locking}. The main point is
that the reduced bilinear form $a(\cdot,\cdot)$ is symmetric, coercive,
local, and satisfies the same local stability and inverse estimates \cite{Fu2025} as the
standard elliptic energy form. Therefore, the cutoff-function argument and
the stable decomposition used in \cite{EricYETY2023,Eric2018,ChungKimZhong2026,chung2025locking} apply without essential
modification.
\end{proof}

\subsection{Explicit multiscale scheme}
In this section, we present the explicit multiscale scheme. Let \( N_T \) be the number of time steps in the temporal mesh grid and \( \tau = T / N_T \) be the time step size. At the time instant \( t_n = n \tau \), we denote the evaluation of the function \( \bm{v} \) at the time instant \( t_n \) by \( \bm{v}^n \), and we denote an approximation of the solution \( \bm{u}(\cdot, t_n) \) by \( \bm{u}_{\rm ms}^n \). For notation simplification, we define
\[
    \partial_{tt}\bm{v}:= \frac{\partial^2\bm{v}}{\partial t^2},\quad \partial_{t}\bm{v}:= \frac{\partial \bm{v}}{\partial t},\quad D_{tt}\bm{v}^n
    :=
    \frac{\bm{v}^{n+1}-2\bm{v}^n+\bm{v}^{n-1}}{\tau^2},
    \quad
    D_t \bm{v}^{n+\frac12}
    :=
    \frac{\bm{v}^{n+1}-\bm{v}^n}{\tau},\quad
    \bar{\bm{v}}^{n+\frac12}:=
    \frac{\bm{v}^{n+1}+\bm{v}^n}{2}.
\]
We recall fine-scale semi-discrete wave equation:
find $\bm{u}_h(t)\in \bm{U}_h$ such that
\begin{equation}
    \label{eq:fine-dynamic}
    (\rho \partial_{tt}\bm{u}_h(t),\bm{v})+a(\bm{u}_h(t),\bm{v})
    =
    (\bm{f}(t),\bm{v})
    \qquad
    \forall \bm{v}\in \bm{U}_h,
\end{equation}
with the initial data satisfying (\ref{discrete_initial_data}). 

Then we derive our fully discrete coarse-scale system by means of Petrov-Galerkin projection of the fine-scale system onto the coarse-scale spaces as follows: for \( n \geq 1 \), find \( \bm{u}_{\mathrm{ms}}^{n+1} \in \bm{V}_{\mathrm{ms}} \) such that
\begin{equation}\label{original_fully_discrete_form}
\left(\rho D_{tt}\bm{u}_{\mathrm{ms}}^n, \bm{w} \right) + a(\bm{u}_{\mathrm{ms}}^n, \bm{w}) = (\bm{f}^n, \bm{w}) \quad \text{for all } \bm{w} \in \bm{V}_\mathrm{aux},
\end{equation}
where the initial data is projected onto the multiscale space \( \bm{V}_{\mathrm{ms}} \) by the following: find \( \bm{u}_{\mathrm{ms}}^0, \bm{u}_{\mathrm{ms}}^1 \in \bm{V}_{\rm ms} \) such that for all \( \bm{w} \in \bm{V}_{\rm ms} \),
\begin{subequations}\label{taylor_initial_data}
\begin{align}
(\rho\bm{u}_{\mathrm{ms}}^0, \bm{w}) &= (\rho\bm{u}_0, \bm{w}), \label{taylor_initial_data_a} \\
(\rho\bm{u}_{\mathrm{ms}}^1, \bm{w}) &= \left(\rho(\bm{u}_0 + \tau \bm{v}_0), \bm{w} \right) + \frac{\tau^2}{2}\left(\bm{f}^0, \pi \bm{w} \right) - \frac{\tau^2}{2} a(\bm{u}_{\mathrm{ms}}^0, \bm{w}).\label{taylor_initial_data_b}
\end{align}
\end{subequations}
To motivate the symmetric formulation, we first consider the global trial space \( \bm{V}_{\rm glo} \). For any \( \bm{v} \in \bm{V}_{\rm glo} \), the weighted \(L^2\)-projection property gives
\[
(\rho(\bm{v}-\pi(\bm{v})),\bm{w})=0
\quad \text{for all } \bm{w}\in\bm{V}_{\rm aux}.
\]
Moreover, \(\bm{v}-\pi(\bm{v})\in\widetilde{\bm{V}}\), by the orthogonality for $a(\cdot,\cdot)$ with respect to $\bm{V}_{\rm glo}$ and $\widetilde{\bm{V}}$, we have
\[
a(\bm{\psi},\bm{v}-\pi(\bm{v}))=0
\quad \text{for all } \bm{\psi},\bm{v}\in\bm{V}_{\rm glo}.
\]
Consequently, the global counterpart of \eqref{original_fully_discrete_form}, obtained by replacing \(\bm{V}_{\rm ms}\) with \(\bm{V}_{\rm glo}\), is exactly equivalent to
\[
b_\rho\!\left(D_{tt}\bm{u}_{\rm glo}^n,\bm{w}\right)
+a(\bm{u}_{\rm glo}^n,\bm{w})
=(\bm{f}^n,\pi\bm{w})
\quad \text{for all } \bm{w}\in\bm{V}_{\rm glo},
\]
where the bilinear form \( b_\rho \) is defined as
\begin{equation}
    \label{eq:brho}
    b_\rho(\bm{v},\bm{w}):=(\rho\pi \bm{v},\pi \bm{w}),
    \qquad \text{for all } \bm{v},\bm{w}\in \bm{U}_h.
\end{equation}
This is the global symmetrization argument similar to \cite[Section~3.3]{Cheung2021}.

For the localized trial space \(\bm{V}_{\rm ms}\), each localized constrained minimizer is \(a\)-orthogonal only to admissible kernel functions supported in the corresponding oversampling region. Therefore, the preceding global \(a\)-orthogonality is not asserted as an exact identity on \(\bm{V}_{\rm ms}\). Following the localized symmetric formulation in \cite[Section~3.4, equation~(3.21)]{Cheung2021}, we take the following projected-mass formulation as the definition of the localized fully discrete multiscale scheme:
\begin{equation}\label{eq:explicit-ms}
b_\rho \left( D_{tt}\bm{u}_{\mathrm{ms}}^n, \bm{w} \right) + a(\bm{u}_{\mathrm{ms}}^n, \bm{w}) = (\bm{f}^n, \pi \bm{w}) \quad \text{for all } \bm{w} \in \bm{V}_{\mathrm{ms}}.
\end{equation}
The associated seminorm is denoted by
\(
    \|\bm{u}\|_{b_\rho}:=b_\rho(\bm{u},\bm{u})^{1/2}
    =
    \|\pi \bm{u}\|_\rho.
\)
Then we clearly have
\[
    b_\rho(\bm{\psi}_{j,m}^i,\bm{\psi}_{j',m}^{i'})
    =
    (\rho \bm{\phi}_j^i,\bm{\phi}_{j'}^{i'})
    =
    \delta_{ii'}\delta_{jj'}.
\]
Therefore, in the multiscale basis coordinates, the coarse mass matrix is the
identity matrix.
Then
\eqref{eq:explicit-ms} gives a fully explicit update:
\[
    \bm{U}^{n+1}
    =
    2\bm{U}^n-\bm{U}^{n-1}
    +
    \tau^2
    \bigl(\bm{F}^n-\bm{K}_{\rm ms}\bm{U}^n\bigr),
\]
where
\((\bm{K}_{\rm ms})_{\alpha\beta}=a(\bm{\psi}_\beta,\bm{\psi}_\alpha),\)
\(\bm{F}_\alpha^n=(\bm{f}^n,\pi \bm{\psi}_\alpha)
\) with \(\bm{\psi}_\beta,\bm{\psi}_\alpha\in \bm{V}_{\mathrm{ms}}\).

\section{Analysis}\label{sec:analysis}
In this section, we provide detailed analyses of the stability and convergence for the explicit multiscale scheme (\ref{eq:explicit-ms}). In particular, we present the inverse Poincaré inequality in Lemma \ref{lem:inverse-poincare}, the discrete energy identity and stability in Lemma \ref{lem:energy-identity}, several elliptic projection estimates in Lemma \ref{lem:elliptic_projection}, temporal projection errors in Lemma \ref{lem:time_projection}, and the final convergence results for both the displacement and the stress in Theorems \ref{estimate_for_u}, \ref{estimate_for_sigma}.

We first state the key inverse Poincaré inequality on the multiscale space, which is important for subsequent analyses.
\begin{lemma}[Inverse Poincaré inequality]
\label{lem:inverse-poincare}
There exists a constant
$\beta>0$ such that
\begin{equation}
    \label{eq:inverse-poincare}
    \|\pi \bm{v}\|_\rho^2
    \ge
    \beta H^2 \|\bm{v}\|_a^2,
    \qquad
    \forall \bm{v}\in \bm{V}_{\rm ms},
\end{equation}
where $\beta$ is dependent on the regularity of the mesh and the eigenvalue $\lambda_{\mathrm{max}}$, but independent of $h$, $H$, and the contrast.
\end{lemma}

\begin{proof}
We first prove the result for the global multiscale space
\(\bm{V}_{\rm glo}\).
By Proposition \ref{right_inverse}, for any
$\bm{q}\in \bm{V}_{\rm aux}$, there exists $R\bm{q}\in \bm{U}_h$ such that
\[
    \pi(R\bm{q})=\bm{q},
    \qquad
    \|R\bm{q}\|_a^2\le C_{\rm lift}H^{-2}\|\bm{q}\|_\rho^2.
\]
Let $\bm{v}\in \bm{V}_{\rm glo}$ and set $\bm{q}=\pi \bm{v}$. Let $\bm{w}\in \bm{U}_h$ be any function satisfying $\pi \bm{w}=\bm{q}$. Since $\bm{q}=\pi \bm{v}$, we
have
\(\pi(\bm{w}-\bm{v})=0.
\)
Thus $\bm{w}-\bm{v}\in\ker\pi$. Using the above orthogonality, we obtain
\[
\begin{aligned}
    a(\bm{w},\bm{w})
    &=
    a(\bm{v}+(\bm{w}-\bm{v}),\bm{v}+(\bm{w}-\bm{v})) =
    a(\bm{v},\bm{v})+2a(\bm{v},\bm{w}-\bm{v})+a(\bm{w}-\bm{v},\bm{w}-\bm{v})  \\
    &=
    a(\bm{v},\bm{v})+a(\bm{w}-\bm{v},\bm{w}-\bm{v})
    \ge
    a(\bm{v},\bm{v}).
\end{aligned}
\]
Therefore,
\(\|\bm{v}\|_a^2
    =
    \min_{\substack{\bm{w}\in \bm{U}_h\\ \pi w=\bm{q}}}
    \|\bm{w}\|_a^2.
\)
Since $R\bm{q}$ is admissible in the above minimization problem, we have
\begin{equation}
\label{global_poin_inequality}
    \|\bm{v}\|_a^2
    \le
    \|R\bm{q}\|_a^2
    \le
    C_{\rm lift}H^{-2}\|\bm{q}\|_\rho^2
    =
    C_{\rm lift}H^{-2}\|\pi \bm{v}\|_\rho^2.
\end{equation}
Equivalently,
\(\|\pi \bm{v}\|_\rho^2
    \ge
    C_{\rm lift}^{-1}H^2\|\bm{v}\|_a^2\) for all \(\bm{v}\in V_{\mathrm{glo}}\).
Let
\(\bm{v}_{\rm ms}
    =
    \sum_{i,j} c_j^i\bm{\psi}_{j,m}^i
    \in \bm{V}_{\rm ms},\)
   \( \bm{V}_{\rm glo}
    =
    \sum_{i,j} c_j^i\bm{\psi}_j^i
    \in \bm{V}_{\rm glo}.\)
By (\ref{global_mini}) and (\ref{local_mini}), we have
\(\pi \bm{v}_{\rm ms} = \pi \bm{v}_{\rm glo}=\sum_{i,j}c_j^i\bm{\phi}_j^i.\)
In terms of Proposition \ref{prop:cem-decay} and the standard finite-overlap argument for the corresponding linear combinations,
\[
    \|\bm{v}_{\rm ms}-\bm{v}_{\rm glo}\|_a
    \le
    \eta_m H^{-1}\|\pi \bm{v}_{\rm ms}\|_\rho,
\]
where $\eta_m$ decays exponentially with respect to $m$. Hence, combining
\eqref{global_poin_inequality} with the above localization estimate, we obtain
\[
\begin{aligned}
    \|\bm{v}_{\rm ms}\|_a
    &\le
    \|\bm{v}_{\rm glo}\|_a+\|\bm{v}_{\rm ms}-\bm{v}_{\rm glo}\|_a  \\
    &\le
    C_{\rm lift}^{1/2}H^{-1}\|\pi \bm{v}_{\rm glo}\|_\rho
    +
    \eta_m H^{-1}\|\pi \bm{v}_{\rm ms}\|_\rho  \\
    &=
    \bigl(C_{\rm lift}^{1/2}+\eta_m\bigr)
    H^{-1}
    \|\pi \bm{v}_{\rm ms}\|_\rho.
\end{aligned}
\]
Therefore,
\(\|\pi \bm{v}_{\rm ms}\|_\rho^2
    \ge
    \bigl(C_{\rm lift}^{1/2}+\eta_m\bigr)^{-2}
    H^2
    \|\bm{v}_{\rm ms}\|_a^2.\)
For properly selected $m$, the constant
\(\beta
    :=
    \bigl(C_{\rm lift}^{1/2}+\eta_m\bigr)^{-2}
\)
is dependent on the regularity of the mesh and the eigenvalue $\lambda_{\mathrm{max}}$, but independent of $h$, $H$, and the contrast. This proves
\eqref{eq:inverse-poincare}.
\end{proof}
As a consequence of Lemma~\ref{lem:inverse-poincare} and the $\rho$-norm
spectral estimate \ref{eq:spectral-estimate},
for every $\bm{v}\in \bm{V}_{\rm ms}$,
\begin{equation}
    \label{eq:rho-full-by-proj}
    \|\bm{v}\|_\rho
    \le
    \|\pi \bm{v}\|_\rho+\|(I-\pi)\bm{v}\|_\rho
    \le
    \left(1+\frac{1}{\sqrt{\Lambda\beta}}\right)
    \|\pi \bm{v}\|_\rho.
\end{equation}
Next we define a discrete total energy which is related to the stability
and convergence of our method. Given a sequence of states $\{\bm{v}^n\}_{n=0}^{N_T}$, we define the discrete total energy at $t=t_{n+\frac{1}{2}}$ by
\begin{equation}
    \label{eq:discrete-energy}
    E_{\pi,\rho}^{n+\frac12}(\bm{v})
    :=
    \frac12
    \|\pi D_t \bm{v}^{n+\frac12}\|_\rho^2
    -
    \frac{\tau^2}{8}
    \|D_t \bm{v}^{n+\frac12}\|_a^2
    +
    \frac12
    \|\bar{\bm{v}}^{n+\frac12}\|_a^2.
\end{equation}
Assume the CFL condition
\begin{equation}
    \label{eq:cfl}
    \frac{\tau}{2H\sqrt{\beta}}<1.
\end{equation}
In terms of Lemma~\ref{lem:inverse-poincare} and the property of the projection operator $\pi$, we have
\begin{equation}
    \label{eq:energy-positive}
    E_{\pi,\rho}^{n+\frac12}(\bm{v})
    \ge
    \frac12
    \left(
        1-\frac{\tau^2}{4\beta H^2}
    \right)
    \|\pi D_t \bm{v}^{n+\frac12}\|_\rho^2
    +
    \frac12
    \|\bar{\bm{v}}^{n+\frac12}\|_a^2.
\end{equation}
Then we give the following discrete energy identity and stability.
\begin{lemma}
\label{lem:energy-identity}
Let $\{\bm{v}^n\}_{n\ge0}\subset \bm{V}_{\rm ms}$ satisfy
\begin{equation}
    \label{eq:abstract-error-eq}
b_\rho(D_{tt}\bm{v}^n,\bm{w})+a(\bm{v}^n,\bm{w})
    =
    (\bm{r}^n, \bm{w}),
    \qquad
    \forall \bm{w}\in \bm{V}_{\rm ms}.
\end{equation}
Then we have
\begin{equation}
    \label{eq:energy-identity}
    E_{\pi,\rho}^{n+\frac12}(\bm{v})
    =
    E_{\pi,\rho}^{\frac12}(\bm{v})
    +
    \tau
    \sum_{k=1}^n
    \left(
        \bm{r}^k,
        \frac{\bm{v}^{k+1}-\bm{v}^{k-1}}{2\tau}
    \right).
\end{equation}
Moreover,
\begin{equation}
    \label{eq:abstract-stability}
    E_{\pi,\rho}^{n+\frac12}(\bm{v})
    \le
    C
    \left[
        E_{\pi,\rho}^{\frac12}(\bm{v})
        +
        \bigl(
            \tau R_1^n
            +
            H\Lambda^{-1/2}R_2^n
        \bigr)^2
    \right],
\end{equation}
where
\[R_1^n=
    \sum_{k=1}^n
    \left\|
        \pi(\rho^{-1}\bm{r}^k)
    \right\|_\rho,\quad R_2^n=
    \left\|
        (I-\pi)(\rho^{-1}\bm{r}^1)
    \right\|_\rho
    +
    \tau
    \sum_{k=1}^{n-1}
    \left\|
        (I-\pi)
        \left(
            \rho^{-1}
            \frac{\bm{r}^{k+1}-\bm{r}^k}{\tau}
        \right)
    \right\|_\rho
    +
    \left\|
        (I-\pi)(\rho^{-1}\bm{r}^n)
    \right\|_\rho.
\]
\end{lemma}
\begin{proof}
Taking
\(\bm{w} = \frac{\bm{v}^{n+1}-\bm{v}^{n-1}}{2\tau}=\frac12 \left(D_t \bm{v}^{n+\frac12}+D_t \bm{v}^{n-\frac12}\right)
\)
in \eqref{eq:abstract-error-eq} and using the linearity of $\pi$, we have
\[
\begin{aligned}
    b_\rho(D_{tt}\bm{v}^n,\bm{w})
    &=
    \left(
        \rho\pi D_{tt}\bm{v}^n,
        \pi \bm{w}
    \right)
    =
    \left(
        \rho
        \frac{
            \pi D_t \bm{v}^{n+\frac12}
            -
            \pi D_t \bm{v}^{n-\frac12}
        }{\tau},
        \frac{
            \pi D_t \bm{v}^{n+\frac12}
            +
            \pi D_t \bm{v}^{n-\frac12}
        }{2}
    \right)
    \\
    &=
    \frac{1}{2\tau}
    \left(
        \|\pi D_t \bm{v}^{n+\frac12}\|_\rho^2
        -
        \|\pi D_t \bm{v}^{n-\frac12}\|_\rho^2
    \right).
\end{aligned}
\]
For the stiffness term, the standard identity gives
\[a(\bm{v}^n,\bm{w})=
    \frac{1}{2\tau}
    \left(\|\bar{\bm{v}}^{n+\frac12}\|_a^2-\|\bar{\bm{v}}^{n-\frac12}\|_a^2
    \right)-\frac{\tau}{8}\left(\|D_t \bm{v}^{n+\frac12}\|_a^2-\|D_t \bm{v}^{n-\frac12}\|_a^2\right).\]
Therefore,
\(
    E_{\pi,\rho}^{n+\frac12}(\bm{v})
    -
    E_{\pi,\rho}^{n-\frac12}(\bm{v})
    =
    \tau
    \left(
        \bm{r}^n,
        \frac{\bm{v}^{n+1}-\bm{v}^{n-1}}{2\tau}
    \right).
\)
Summing this identity from $1$ to $n$ proves
\eqref{eq:energy-identity}.
It remains to bound the right-hand side. Let
\(\bm{z}^k:=\frac{\bm{v}^{k+1}-\bm{v}^{k-1}}{2\tau}=\frac{\bar{\bm{v}}^{k+\frac12}-\bar{\bm{v}}^{k-\frac12}}{\tau}
\)
and we write
\((\bm{r}^k,\bm{z}^k)=\left(\rho(\rho^{-1}\bm{r}^k),\bm{z}^k\right).
\)
Using the $\rho$-orthogonality of $\pi$, we decompose
\[
\begin{aligned}
    (\bm{r}^k,\bm{z}^k)
    =
    \left(
        \rho\pi(\rho^{-1}\bm{r}^k),
        \pi \bm{z}^k
    \right)
    +
    \left(
        \rho(I-\pi)(\rho^{-1}\bm{r}^k),
        (I-\pi)\bm{z}^k
    \right).
\end{aligned}
\]
By \eqref{eq:energy-positive}, we have
\(\tau
    \sum_{k=1}^n
    \left|
        \left(
            \rho\pi(\rho^{-1}\bm{r}^k),
            \pi \bm{z}^k
        \right)
    \right|\le
    \tau
    \sum_{k=1}^n
    \left\|
        \pi(\rho^{-1}\bm{r}^k)
    \right\|_\rho
    \|\pi \bm{z}^k\|_\rho.
\)
Moreover, we directly have
\( \pi \bm{z}^k=\frac12\left(\pi D_t \bm{v}^{k+\frac12}+\pi D_t \bm{v}^{k-\frac12}
\right),
\)
and hence \eqref{eq:energy-positive} implies
\[
    \|\pi \bm{z}^k\|_\rho
    \le
    C
    \max_{0\le \ell\le n}
    \left(
        E_{\pi,\rho}^{\ell+\frac12}(\bm{v})
    \right)^{1/2}.
\]
Therefore,
\begin{equation}\label{ineq_1}
\begin{aligned}
    \tau
    \sum_{k=1}^n
    \left|
        \left(
            \rho\pi(\rho^{-1}\bm{r}^k),
            \pi \bm{z}^k
        \right)
    \right|
    &\le
    C
    \tau R_1^n
    \max_{0\le \ell\le n}
    \left(
        E_{\pi,\rho}^{\ell+\frac12}(\bm{v})
    \right)^{1/2}.
\end{aligned}
\end{equation}
Setting
\(\bm{q}^k=(I-\pi)(\rho^{-1}\bm{r}^k)\),
then $(\rho\bm{q}^k, \bm{w})=0$ for all $\bm{w}\in \bm{V}_{\rm aux}$.
Using summation by parts,
\[\tau\sum_{k=1}^n(\rho \bm{q}^k,\bm{z}^k)=\sum_{k=1}^n
    \left(
        \rho \bm{q}^k,
        \bar{\bm{v}}^{k+\frac12}-\bar{\bm{v}}^{k-\frac12}
    \right)=
    (\rho \bm{q}^n,\bar{\bm{v}}^{n+\frac12})
    -
    (\rho \bm{q}^1,\bar{\bm{v}}^{\frac12})
    -
    \sum_{k=1}^{n-1}
    \left(
        \rho(\bm{q}^{k+1}-\bm{q}^k),
        \bar{\bm{v}}^{k+\frac12}
    \right).\]
Since $(\rho\bm{q}^k, \bm{w})=0$ for all $\bm{w}\in \bm{V}_{\rm aux}$, then for any
$\bar{\bm{v}}^{k+\frac12}$, we have
\((\rho \bm{q}^k,\bar{\bm{v}}^{k+\frac12})
    =
    \left(
        \rho \bm{q}^k,
        (I-\pi)\bar{\bm{v}}^{k+\frac12}
    \right).
\)
Combining the $\rho$-norm spectral estimate (\ref{eq:spectral-estimate}),
we obtain
\[\left|\tau\sum_{k=1}^n(\rho \bm{q}^k,\bm{z}^k)
    \right|\le
    H\Lambda^{-1/2}
    \biggl[
        \|\bm{q}^1\|_\rho
        +
        \sum_{k=1}^{n-1}
        \|\bm{q}^{k+1}-\bm{q}^k\|_\rho
        +
        \|\bm{q}^n\|_\rho
    \biggr]\cdot
    \max_{0\le \ell\le n}
    \left(
        E_{\pi,\rho}^{\ell+\frac12}(\bm{v})
    \right)^{1/2}.\]
Since
\(\bm{q}^{k+1}-\bm{q}^k=(I-\pi)\left(\rho^{-1}(\bm{r}^{k+1}-\bm{r}^k)\right),\)
we have
\[
    \sum_{k=1}^{n-1}
    \|\bm{q}^{k+1}-\bm{q}^k\|_\rho
    =
    \tau
    \sum_{k=1}^{n-1}
    \left\|
        (I-\pi)
        \left(
            \rho^{-1}
            \frac{\bm{r}^{k+1}-\bm{r}^k}{\tau}
        \right)
    \right\|_\rho.
\]
Therefore, by the definition of $R_2^n$,
\begin{equation}\label{ineq_2}
\begin{aligned}
    \left|
    \tau
    \sum_{k=1}^n
    (\rho \bm{q}^k,\bm{z}^k)
    \right|
    &\le
    H\Lambda^{-1/2}
    R_2^n
    \max_{0\le \ell\le n}
    \left(
        E_{\pi,\rho}^{\ell+\frac12}(\bm{v})
    \right)^{1/2}.
\end{aligned}
\end{equation}
Combining (\ref{ineq_1}), (\ref{ineq_2}) and
\eqref{eq:energy-identity}, we obtain
\[
    E_{\pi,\rho}^{n+\frac12}(\bm{v})
    \le
    E_{\pi,\rho}^{\frac12}(\bm{v})
    +
    C
    \left(
        \tau R_1^n
        +
        H\Lambda^{-1/2}R_2^n
    \right)
    \max_{0\le \ell\le n}
    \left(
        E_{\pi,\rho}^{\ell+\frac12}(\bm{v})
    \right)^{1/2}.
\]
Taking the maximum over $0\le n\le N$ and applying Young's inequality yields \eqref{eq:abstract-stability}. This completes the proof.
\end{proof}
Next we give the stability of the explicit multiscale scheme (\ref{eq:explicit-ms}).
\begin{theorem}
\label{thm:ms-stability}
Let $\{\bm{u}_{\rm ms}^n\}\subset \bm{V}_{\rm ms}$ solve \eqref{eq:explicit-ms}.
Assume the CFL condition \eqref{eq:cfl}. Then
\begin{equation}
    \label{eq:ms-stability}
    \|D_t\bm{u}_{\rm ms}^{n+\frac12}\|_\rho^2
    +
    \|\bar{\bm{u}}_{\rm ms}^{n+\frac12}\|_a^2
    \le
    C
    \left[
        E^{\frac12}_{\pi,\rho}(\bm{u}_{\rm ms})
        +
        \tau^2
        \left(
            \sum_{k=1}^n
            \|\pi (\rho^{-1}\bm{f}^k)\|_\rho
        \right)^2
    \right].
\end{equation}
\end{theorem}

\begin{proof}
Recall the multiscale scheme \eqref{eq:explicit-ms}:
\[
    b_\rho(D_{tt}\bm{u}_{\rm ms}^n,\bm{w})+a(\bm{u}_{\rm ms}^n,\bm{w})=
    (\rho(\rho^{-1}\bm{f}^n),\pi \bm{w})
    =
    (\rho \pi(\rho^{-1}\bm{f}^n),\pi \bm{w})=(\rho \pi(\rho^{-1}\bm{f}^n), \bm{w}).
\]
Letting $\bm{r}^n=\rho \pi(\rho^{-1}\bm{f}^n)$, by Lemma \ref{lem:energy-identity}, we know $R_2^n=0$ and can obtain that
\[
    E^{n+\frac12}_{\pi,\rho}(\bm{u}_{\rm ms})
    \le
    C
    \left[
        E^{\frac12}_{\pi,\rho}(\bm{u}_{\rm ms})
        +
        \tau^2
        \left(
            \sum_{k=1}^n
            \|\pi (\rho^{-1}\bm{f}^k)\|_\rho
        \right)^2
    \right].
\]
Using the lower bound \eqref{eq:energy-positive}, we get
\[
    \|\pi D_t\bm{u}_{\rm ms}^{n+\frac12}\|_\rho^2
    +
    \|\bar{\bm{u}}_{\rm ms}^{n+\frac12}\|_a^2
    \le
    C
    \left[
        E^{\frac12}_{\pi,\rho}(\bm{u}_{\rm ms})
        +
        \tau^2
        \left(
            \sum_{k=1}^n
            \|\pi (\rho^{-1}\bm{f}^k)\|_\rho
        \right)^2
    \right].
\]
Finally, \eqref{eq:rho-full-by-proj} gives
\(\|D_t\bm{u}_{\rm ms}^{n+\frac12}\|_\rho
    \le
    C\|\pi D_t\bm{u}_{\rm ms}^{n+\frac12}\|_\rho.\)
This proves \eqref{eq:ms-stability}.
\end{proof}
To obtain the final convergence result, we first provide some elliptic projection estimates.
Define the fine-scale elliptic solution operator
\(
G_h : \bm{U}_h \to \bm{U}_h
\)
by
\begin{equation}
a(G_h \bm{g}, \bm{w}) = (\rho \bm{g}, \bm{w}), \quad \forall \bm{w} \in \bm{U}_h.
\end{equation}
Note that the domain of $G_h$ can be extended to $\bm{L}^2(\Omega)$. Let $P_H : \bm{U}_h \to \bm{V}_{\rm ms}$ be the $a$-orthogonal projection:
\begin{equation}
a(P_H v, \bm{w}) = a(\bm{v}, \bm{w}), \quad \forall \bm{w} \in \bm{V}_{\rm ms}.
\end{equation}
Let $P_\mathrm{glo} : \bm{U}_h \to V_{\mathrm{glo}}$ be the $a$-orthogonal projection for the global multiscale space:
\begin{equation}\label{global_projection}
a(P_\mathrm{glo} v, \bm{w}) = a(\bm{v}, \bm{w}), \quad \forall \bm{w} \in V_{\mathrm{glo}}.
\end{equation}

\begin{lemma}
\label{lem:elliptic_projection}
Assume that the oversampling size $m$ is selected appropriately, ensuring that the right-hand side of (\ref{decay_inequality}) is sufficiently small. Then the following $a$-norm and $\rho$-norm estimates hold:
\[
\|(I - P_H)  G_h \bm{g}\|_a \le C \Lambda^{-1/2} H \|\bm{g}\|_\rho, \quad \forall \bm{g} \in \bm{U}_h.
\]
Moreover, by a duality argument,
\[
\|(I - P_H)  G_h \bm{g}\|_\rho \le C \Lambda^{-1} H^2 \|\bm{g}\|_\rho, \quad \forall \bm{g} \in \bm{U}_h.
\]
\end{lemma}

\begin{proof}
By the definition of $P_\mathrm{glo}$ and the orthogonal property of $\bm{V}_\mathrm{glo},\widetilde{\bm{V}}$ with respect to $a-$norm, we have
\(
\pi (I - P_{\mathrm{glo}})  G_h \bm{g} = 0.
\)
Therefore, by the spectral estimate \eqref{eq:spectral-estimate},
\[
\|(I - P_{\mathrm{glo}})  G_h \bm{g}\|_\rho = \|(I - \pi) (I - P_{\mathrm{glo}})  G_h \bm{g}\|_\rho \le \Lambda^{-1/2} H \|(I - P_{\mathrm{glo}})  G_h \bm{g}\|_a.
\]
By (\ref{global_projection}), we have
\(
\|(I - P_{\mathrm{glo}})  G_h \bm{g}\|_a^2 = a((I - P_{\mathrm{glo}})  G_h \bm{g}, (I - P_{\mathrm{glo}})  G_h \bm{g}) = a( G_h \bm{g}, (I - P_{\mathrm{glo}})  G_h \bm{g}) = (\rho \bm{g}, (I - P_{\mathrm{glo}})  G_h \bm{g}).
\)
Consequently,
\[
\|(I - P_{\mathrm{glo}})  G_h \bm{g}\|_a^2 \le \|\bm{g}\|_\rho \|(I - P_{\mathrm{glo}})  G_h \bm{g}\|_\rho \le \Lambda^{-1/2} H \|\bm{g}\|_\rho \|(I - P_{\mathrm{glo}})  G_h \bm{g}\|_a,
\]
which gives
\begin{equation} \label{proj_est}
\|(I - P_{\mathrm{glo}})  G_h \bm{g}\|_a \le \Lambda^{-1/2} H \|\bm{g}\|_\rho.
\end{equation}
For the multiscale space $\bm{V}_{\rm ms}$, let $\bm{v}_{\mathrm{ms}} \in \bm{V}_{\rm ms}$ be the localized counterpart of $P_{\mathrm{glo}}  G_h \bm{g}$. That is, suppose $P_{\mathrm{glo}}  G_h \bm{g}=\sum_{i=1}^N\sum_{j=1}^{l_i}c^i_j\bm{\psi}^i_j$, then we take $\bm{v}_{\mathrm{ms}}=\sum_{i=1}^N\sum_{j=1}^{l_i}c^i_j\bm{\psi}^i_{j,m}$. 
By Proposition \ref{prop:cem-decay}, together with the standard finite-overlap
argument for linear combinations of localized CEM basis functions
(see also \cite{EricYETY2023,Eric2018,ChungKimZhong2026,chung2025locking}),
we obtain
\begin{equation}\label{local_est}
\norm{P_{\mathrm{glo}}G_h\bm{g}-\bm{v}_{\mathrm{ms}}}_a
\leq
C_{\mathrm{loc}}\theta^m\norm{\bm{g}}_\rho .
\end{equation}
Here $0<\theta<1$.
For proper $m$, this term is bounded by $C \Lambda^{-1/2} H \|\bm{g}\|_\rho$. 
By the property of $P_H$, we know
\begin{align}
\|(I - P_H)  G_h \bm{g}\|_a^2&=a((I - P_H)  G_h \bm{g},(I - P_H)  G_h \bm{g})=a((I - P_H)  G_h \bm{g}, G_h \bm{g}) \nonumber \\
&=a((I - P_H)  G_h \bm{g}, G_h \bm{g}-\bm{v}_\mathrm{ms})\leq \|(I - P_H)  G_h \bm{g}\|_a \| G_h \bm{g}-\bm{v}_\mathrm{ms}\|_a. \nonumber
\end{align}
Then combining (\ref{proj_est}) and (\ref{local_est}),  we obtain
\begin{align}
\|(I - P_H)  G_h \bm{g}\|_a \le \| G_h \bm{g} - \bm{v}_{\mathrm{ms}}\|_a
\le \|(I - P_{\mathrm{glo}})  G_h \bm{g}\|_a + \|P_{\mathrm{glo}}  G_h \bm{g} - \bm{v}_{\mathrm{ms}}\|_a
\le C \Lambda^{-1/2} H \|\bm{g}\|_\rho.
\end{align}
This proves the $a$-norm estimate.

Next we prove the $\rho$-norm estimate. Let
\(
\bm{e} = (I - P_H)  G_h \bm{g}.
\)
By duality,
\(
\|\bm{e}\|_\rho^2 = (\rho \bm{e}, \bm{e}) = a(G_h \bm{e}, \bm{e}).
\)
Since $P_H$ is the $a$-orthogonal projection, i.e., 
\(
a(P_H G_h \bm{e}, \bm{e}) = 0,
\)
then
\[
\|\bm{e}\|_\rho^2 = a((I - P_H) G_h \bm{e}, \bm{e}) \le \|(I - P_H) G_h \bm{e}\|_a \|\bm{e}\|_a.
\]
Using the $a$-norm estimate just proved twice gives
\[
\|(I - P_H) G_h \bm{e}\|_a \le C \Lambda^{-1/2} H \|\bm{e}\|_\rho
\]
and
\[
\|\bm{e}\|_a = \|(I - P_H)  G_h \bm{g}\|_a \le C \Lambda^{-1/2} H \|\bm{g}\|_\rho.
\]
Therefore, we obtain
\(
\|\bm{e}\|_\rho^2 \le C \Lambda^{-1} H^2 \|\bm{e}\|_\rho \|\bm{g}\|_\rho.
\)
This completes the proof.
\end{proof}
Let \(\bm{\theta}^n := (I - P_H) \bm{u}_h(t_n)\). Then, in view of the projection estimates established in Lemma \ref{lem:elliptic_projection}, we obtain the following projection error bound in time.
\begin{lemma}
\label{lem:time_projection}
Let $\bm{u}_h(t)$ solve \eqref{eq:reduced-wave} and define
\(
\bm{g}(t) := \rho^{-1}\bm{f}(t) - \partial_{tt} \bm{u}_h(t).
\) Assume \( \bm{g} \in C^4([0, T]; \bm{L}^2(\Omega))\).
Then
\(
\bm{u}_h(t) = G_h \bm{g}(t)
\) and 
\[
\|\bm{\theta}^n\|_\rho \le C \Lambda^{-1} H^2 \|\bm{g}(t_n)\|_\rho.
\]
Moreover, define the norm \(\|\bm{v}\|_{L^\infty(0,T;\rho)} 
:= \operatorname*{ess\,sup}_{0 < t < T} \|\bm{v}(t)\|_{\rho}
\). Then we have
\begin{equation} \label{2-derivative-estimate}
\|D_{tt} \bm{\theta}^n\|_\rho \le C \Lambda^{-1} H^2
\big( \|\partial_{tt} \bm{g}\|_{L^\infty(0,T;\rho)} + \tau^2 \|\partial_{tttt} \bm{g}\|_{L^\infty(0,T;\rho)} \big).
\end{equation}
\end{lemma}

\begin{proof}
From \eqref{eq:reduced-wave},
\[
a(\bm{u}_h(t), \bm{v}) = (\rho(\rho^{-1}\bm{f}(t) - \partial_{tt} \bm{u}_h(t)), \bm{v}) = (\rho \bm{g}(t), \bm{v}), \quad \forall \bm{v} \in \bm{U}_h.
\]
Hence $\bm{u}_h(t) = G_h \bm{g}(t)$. Therefore,
\(
\bm{\theta}^n = (I - P_H) G_h \bm{g}(t_n).
\)
Applying Lemma~\ref{lem:elliptic_projection} gives
\[
\|\bm{\theta}^n\|_\rho \le C \Lambda^{-1} H^2 \|\bm{g}(t_n)\|_\rho.
\]
Since $P_H$ and $G_h$ are time-independent linear operators, we have
\(
D_{tt} \bm{\theta}^n = (I - P_H) G_h (D_{tt} \bm{g})^n,
\)
where
\(
(D_{tt} \bm{g})^n = \frac{\bm{g}(t_{n+1}) - 2\bm{g}(t_n) + g(t_{n-1})}{\tau^2}.
\)
Again by Lemma~\ref{lem:elliptic_projection},
\[
\|D_{tt} \bm{\theta}^n\|_\rho \le C \Lambda^{-1} H^2 \|(D_{tt} \bm{g})^n\|_\rho.
\]
Taylor expansion gives
\(
\|(D_{tt} \bm{g})^n\|_\rho \le C (\|\partial_{tt} \bm{g}\|_{L^\infty(0,T;\rho)} + \tau^2 \|\partial_{tttt} \bm{g}\|_{L^\infty(0,T;\rho)}).
\)
This completes the proof.
\end{proof}

Let
\(
\bm{\varepsilon}^n := \bm{u}_h^n - \bm{u}^n_{\mathrm{ms}}.
\)
We split the error as
\(
\bm{\varepsilon}^n = \bm{\theta}^n - \bm{\delta}^n,
\)
where
\(
\bm{\theta}^n := (I - P_H) \bm{u}_h^n,\) \(\bm{\delta}^n := \bm{u}^n_{\mathrm{ms}} - P_H \bm{u}_h^n.
\) Then we give the following energy error estimate.

\begin{theorem}
\label{thm:energy_error}
Assume the CFL condition \eqref{eq:cfl}. Let
$\bm{q}:=\rho^{-1}\bm{f}$ and
$\bm{g}:=\bm{q}-\partial_{tt}\bm{u}_h$.
Assume that
$\bm{q}\in C^1([0,T];\bm{H}^1(\Omega))$,
$\bm{g}\in C^4([0,T];\bm{L}^2(\Omega))$,
and
$\bm{u}_h\in C^4([0,T];\bm{L}^2(\Omega))$,
with $\partial_{tt}\bm{u}_h$ and
$\partial_{ttt}\bm{u}_h$ uniformly bounded in the energy norm.
Assume that the corresponding regularity bounds are independent of
$h$, $H$, and $\tau$. Then
\begin{equation}
    \|D_t \bm{\delta}^{n+\frac12}\|_\rho
    +
    \|\bar{\bm{\delta}}^{n+\frac12}\|_a
    \le C\left(\Lambda^{-1}H+\tau^2\right).
\end{equation}
\end{theorem}
\begin{proof}
We derive an equation for $\bm{\delta}^n$. Since $P_H$ is the $a$-orthogonal projection,
\begin{equation*}
a(P_H \bm{u}_h(t_n), \bm{w}) = a(\bm{u}_h(t_n), \bm{w}), \quad \forall \bm{w} \in \bm{V}_{\rm ms}.
\end{equation*}
Subtracting this projected fine-scale equation (\ref{eq:fine-dynamic}) from the multiscale scheme (\ref{eq:explicit-ms}) gives
\begin{equation}\label{es_2}
b_\rho(D_{tt} \bm{\delta}^n, \bm{w}) + a(\bm{\delta}^n, \bm{w}) = (\bm{r}^n, \bm{w}), \quad \forall \bm{w} \in \bm{V}_{\rm ms}.
\end{equation}
We decompose the residual into four parts:
\(
\bm{r}^n = \bm{r}^n_1 + \bm{r}^n_2 + \bm{r}^n_3 + \bm{r}^n_4,
\)
where
\begin{equation*}
\bm{r}^n_1 = -\rho(I-\pi)(\rho^{-1}\bm{f}^n),
\quad
\bm{r}^n_2 = \rho\pi( \partial_{tt} \bm{u}_h(t_n) - D_{tt} \bm{u}_h(t_n) ),
\end{equation*}and
\begin{equation*}
\bm{r}^n_3 = \rho(I-\pi)\left( \partial_{tt}\bm{u}_{h}(t_{n}) \right),\quad \bm{r}^n_4 = \rho\pi(D_{tt} \bm{\theta}^n).
\end{equation*}
\noindent
By Lemma \ref{lem:energy-identity} and inequality (\ref{eq:rho-full-by-proj}), we have:
\begin{equation}\label{es_1}
\left\| D_t\bm{\delta}^{n+\frac{1}{2}} \right\|_{\rho}^2
+ \|\bar{\bm{\delta}}^{n+\frac12}\|_a^2 \leq C \left( E^{\frac12}_{\pi,\rho}(\bm{\delta}) + \big( \tau R_1^n + \Lambda^{-\frac12} H R_2^n \big)^2 \right),
\end{equation}
where 
\[
\begin{aligned}
R_1^n := &\; \sum_{k=1}^{n} \left\| \pi \left(D_{tt} \bm{\theta}^k \right) \right\|_{\rho} + \sum_{k=1}^{n} \left\| \pi \left( \partial_{tt} \bm{u}_h(t_k) - D_{tt} \bm{u}_h(t_k) \right) \right\|_{\rho},
\end{aligned}
\]
\[
\begin{aligned}
R_2^n := &\; \| (I-\pi)(\rho^{-1}\bm{f}^{1}) \|_{\rho}
+ \tau \sum_{k=1}^{n-1} \left\| (I-\pi)\left(\rho^{-1} \frac{\bm{f}^{k+1} - \bm{f}^{k}}{\tau} \right) \right\|_{\rho}
+ \| (I-\pi)(\rho^{-1}\bm{f}^{n}) \|_{\rho} \\
&+ \|(I-\pi)\left( \partial_{tt}\bm{u}_{h}(t_{1}) \right) \|_{\rho}
+ \tau \sum_{k=1}^{n-1} \left\| (I-\pi)\frac{\partial_{tt}\bm{u}_{h}(t_{k+1})  - \partial_{tt}\bm{u}_{h}(t_{k}) }{\tau} \right\|_{\rho} + \| (I-\pi)\left( \partial_{tt}\bm{u}_{h}(t_{n}) \right) \|_{\rho}.
\end{aligned}
\]

\noindent
We then estimate each part.

\textbf{Step 1. Initial energy estimate.}  
First, from (\ref{eq:discrete-energy}) we have
\[
E^{\frac{1}{2}}(\bm{\delta}) \le \frac{1}{2}\|\pi D_t\bm{\delta}^{\frac{1}{2}}\|_{\rho}^{2}
+ \frac{1}{2} \|\overline{\bm{\delta}}^{\frac{1}{2}}\|_{a}^{2}.
\]
By the triangle inequality,
\(
\|\overline{\bm{\delta}}^{\frac{1}{2}}\|_{a} \le \frac{1}{2}\bigl(\|\bm{\delta}^{0}\|_{a}+\|\bm{\delta}^{1}\|_{a}\bigr).
\)
Note that $\bm{u}_{\rm ms}^0$ is the $\rho$-weighted
$L^2$-projection of $\bm{u}_h^0$ onto $V_{\rm ms}$. Indeed, since
$\bm{V}_{\rm ms}\subset \bm{U}_h$, (\ref{discrete_initial_data_a}) and (\ref{taylor_initial_data_a}) imply
\((\rho(\bm{u}_{\rm ms}^0-\bm{u}_h^0),\bm{w})=0,
\) for all \(\bm{w}\in \bm{V}_{\rm ms}.\)
Since $\bm{\delta}^0=\bm{u}_{\rm ms}^0-P_H\bm{u}_h^0\in \bm{V}_{\rm ms}$,
combining the above orthogonality and Lemma~\ref{lem:elliptic_projection} (with $\bm{g} = \rho^{-1}\bm{f}(0) - \partial_{tt} \bm{u}_h(0)$) yields
\[
    \|\bm{\delta}^0\|_\rho
    \le
    \|\bm{u}_h^0-P_H\bm{u}_h^0\|_\rho
    \le C\Lambda^{-1}H^2 ,\qquad
\|\bm{u}_{h}^{0}-P_{H}\bm{u}_{h}^{0}\|_{a} \le C\Lambda^{-1/2}H.
\]
Using the inverse inequality (\ref{eq:inverse-poincare}), we further obtain
\[
    \|\bm{\delta}^0\|_a
    \le CH^{-1}\|\pi\bm{\delta}^0\|_\rho
    \le C\Lambda^{-1}H.
\]
From the definitions of $\bm{u}_{\mathrm{ms}}^{1}$ and $\bm{u}_{h}^{1}$ (\ref{discrete_initial_data_b}), (\ref{taylor_initial_data_b}) and the Taylor extensions, one obtains (using the similar analyses to \cite[Lemma 4.9]{Cheung2021})
\[
\|\bm{\delta}^{1}-\bm{\delta}^{0}\|_{\rho} \le \|\bm{\theta}^{1}-\bm{\theta}^{0}\|_{\rho}+\|\bm{\varepsilon}^{1}-\bm{\varepsilon}^{0}\|_{\rho} \le C\tau\bigl(\Lambda^{-1}H^{2}+\tau^{2}\bigr),
\]
where the term $\Lambda^{-1}H^{2}$ originates from the elliptic projection error (Lemma~\ref{lem:elliptic_projection} and (\ref{2-derivative-estimate})) and the term $\tau^{2}$ from the second-order time discretization. Then 
\begin{equation}\label{D_tdelta_estimate}
\|\pi D_t\bm{\delta}^{\frac{1}{2}}\|_{\rho}^2=\bigl\|\pi\bigl(\frac{\bm{\delta}^{1}-\bm{\delta}^{0}}{\tau}\bigr)\bigr\|_{\rho}^2\le \bigl\|\frac{\bm{\delta}^{1}-\bm{\delta}^{0}}{\tau}\bigr\|_{\rho}\le C(\Lambda^{-1}H^2+\tau^2)^2.
\end{equation}
Using again the inverse inequality (\ref{eq:inverse-poincare}),
\[
\|\bm{\delta}^{1}-\bm{\delta}^{0}\|_{a} \le C H^{-1}\|\bm{\delta}^{1}-\bm{\delta}^{0}\|_{\rho}
\le C\tau \bigl(\Lambda^{-1}H + \tau^{2}/H\bigr).
\]
Using $\tau\lesssim H$ from the CFL condition (\ref{eq:cfl}), this becomes $\|\bm{\delta}^{1}-\bm{\delta}^{0}\|_{a}\le C(\Lambda^{-1}H^2+\tau^2)$.  
Then we have
\(
\|\bm{\delta}^{1}\|_{a} \le \|\bm{\delta}^{0}\|_{a}+\|\bm{\delta}^{1}-\bm{\delta}^{0}\|_{a}
\le C\bigl(\Lambda^{-1}H+\tau^2\bigr).
\)
Thus,
\[
E^{\frac{1}{2}}(\bm{\delta}) \le C\bigl(\Lambda^{-1}H+\tau^{2}\bigr)^2.
\]

\textbf{Step 2. Estimate of $R_1^n$ and $R_2^n$.}  
Using (\ref{2-derivative-estimate}) and Taylor extensions:
\[
\sum_{k=1}^{n} \left\| \pi \left(D_{tt} \bm{\theta}^k \right) \right\|_{\rho}
\leq C \tau^{-1} \Lambda^{-1} H^2,
\quad
\sum_{k=1}^{n} \left\| \pi \left( \partial_{tt} \bm{u}_h(t_k) - D_{tt} \bm{u}_h(t_k) \right) \right\|_{\rho}
\leq C \tau \left\| \frac{\partial^4 \bm{u}_h}{\partial t^4} \right\|_{C([0,T];L^2)}.
\]
Thus we have
\(
R_1^n \leq C \tau^{-1} \Lambda^{-1} H^2 + C \tau.
\)
By (\ref{eq:spectral-estimate}) and the regularity of $u_h$ with respect to $t$, we have
\[
\|(I-\pi)\left( \partial_{tt}\bm{u}_{h}(t_k) \right) \|_{\rho}
\leq C \Lambda^{-1/2} H \left\| \partial_{tt}\bm{u}_{h}(t_k) \right\|_a
\leq C \Lambda^{-1/2} H.
\]
For the time difference term,
by the mean value theorem, there exists $\xi_k \in (t_k, t_{k+1})$ such that
\[
\frac{ \partial^2_{tt} \bm{u}_h(t_{k+1}) - \partial^2_{tt} \bm{u}_h(t_k) }{\tau}
= \frac{\partial^3 \bm{u}_h}{\partial t^3}(\xi_k),
\]
Hence,
\[
\left\| (I-\pi)\frac{\partial_{tt}\bm{u}_{h}(t_{k+1})  - \partial_{tt}\bm{u}_{h}(t_{k})}{\tau} \right\|_{\rho}
\leq C \Lambda^{-1/2} H.
\]
Summing over $k=1$ to $n-1$ and multiplying by $\tau$, we obtain
\[
\tau \sum_{k=1}^{n-1} \left\| (I-\pi)\frac{\partial_{tt}\bm{u}_{h}(t_{k+1})  - \partial_{tt}\bm{u}_{h}(t_{k}) }{\tau} \right\|_{\rho}
\leq T \cdot C \Lambda^{-1/2} H = O(\Lambda^{-1/2} H).
\]
By (\ref{eq:spectral-estimate}) and the assumed spatial regularity of $\bm{q}$ (recall that $\bm{q}=\rho^{-1}\bm{f}$), we know
\[
\| (I-\pi)(\bm{q}^k) \|_{\rho} \leq \Lambda^{-1/2} H \| \bm{q} \|_{C([0,T]; H^1)},\quad
\tau \sum_{k=1}^{n-1} \left\| (I-\pi)\left( \frac{\bm{q}^{k+1} - \bm{q}^{k}}{\tau} \right) \right\|_{\rho}
\leq C \Lambda^{-1/2} H \left\| \partial_t \bm{q} \right\|_{C([0,T]; H^1)}.
\]
Thus, we have $R_2^n\leq C\Lambda^{-1/2} H$. Substitute the estimates into the key quantity:
\[
\begin{aligned}
\tau R_1^n + \Lambda^{-1/2} H R_2^n
&\leq \tau \left( C \tau^{-1} \Lambda^{-1} H^2 + C \tau \right)
+ \Lambda^{-1/2} H \cdot C\Lambda^{-1/2} H \\
&\leq C \Lambda^{-1} H^2 + C \tau^2 + C \Lambda^{-1} H^2\leq C(\Lambda^{-1} H^2 + \tau^2).
\end{aligned}
\]

\textbf{Step 3. Final combination.}  
Applying the estimate of $E^{\frac{1}{2}}_{\pi,\rho}(\bm{\delta})$ and the bound on $\tau R_1^n + \Lambda^{-1/2} H R_2^n$ into (\ref{es_1}), we obtain:
\[
\quad \|D_t \bm{\delta}^{n+\frac12}\|_\rho^2 + \|\bar{\bm{\delta}}^{n+\frac12}\|_a^2
\leq C \left( \Lambda^{-1} H^2 + \tau^2 \right)^2+C\bigl(\Lambda^{-1}H+\tau^{2}\bigr)^2\leq C\bigl(\Lambda^{-1}H+\tau^{2}\bigr)^2.
\]
This completes the proof.
\end{proof}
Next we provide a $\rho$-norm error estimate for the displacement.
\begin{theorem}\label{estimate_for_u}
Under the assumptions of Theorem~\ref{thm:energy_error}, there exists a constant
$C>0$, independent of $h$, $H$, and $\tau$, such that
\[
\max_{0\le n\le N_T-1}
\left\|
\frac{
\bm{u}_h(t_{n+1})-\bm{u}^{n+1}_{\mathrm{ms}}
+
\bm{u}_h(t_n)-\bm{u}^n_{\mathrm{ms}}
}{2}
\right\|_{\rho}
\le
C\bigl(\Lambda^{-1}H^2+\tau^2\bigr).
\]
Equivalently,
\(
\max_{0\le n\le N_T-1}
\|\bar{\bm{\varepsilon}}^{\,n+\frac12}\|_\rho
\le
C\bigl(\Lambda^{-1}H^2+\tau^2\bigr),
\)
where
\(
\bm{\varepsilon}^n=\bm{u}_h(t_n)-\bm{u}_{\mathrm{ms}}^n,\)
\(\bar{\bm{\varepsilon}}^{\,n+\frac12}
=
\frac{\bm{\varepsilon}^{n+1}+\bm{\varepsilon}^n}{2}.
\)
\end{theorem}

\begin{proof}

Recall the error decomposition
\(
\bm{\varepsilon}^n
=
\bm{\theta}^n-\bm{\delta}^n,
\)
where
\(
\bm{\theta}^n=(I-P_H)\bm{u}_h(t_n),\)
\(
\bm{\delta}^n=u_{\mathrm{ms}}^n-P_H\bm{u}_h(t_n).
\)
Hence
\(
\bar{\bm{\varepsilon}}^{\,n+\frac12}
=
\bar{\bm{\theta}}^{\,n+\frac12}
-
\bar{\bm{\delta}}^{\,n+\frac12}.
\)
By Lemma~\ref{lem:time_projection},
\(
\|\bar{\bm{\theta}}^{\,n+\frac12}\|_\rho
\le
C\Lambda^{-1}H^2.
\)
Therefore it remains to estimate
$\bar{\bm{\delta}}^{\,n+\frac12}$.

\paragraph{Step 1. Estimate of the projected error.}

We use the standard discrete Baker argument; see also
\cite[Theorem~4.11]{Cheung2021}.  The error equation derived in
Theorem~\ref{thm:energy_error} reads
\[
b_\rho(D_{tt}\bm{\delta}^n,\bm{w})
+a(\bm{\delta}^n,\bm{w})
=(\bm r^n,\bm w),
\qquad
\forall \bm w\in\bm V_{\rm ms}.
\]
Define
\(
\Delta^n:=\tau\sum_{k=1}^n\bm\delta^k,\)
\(\Delta^0:=0.
\)
For each \(j\ge1\), summing the error equation from \(k=1\) to \(j\)
and then taking
\(
\bm w=\Delta^{j+1}-\Delta^{j-1}
      =\tau(\bm\delta^{j+1}+\bm\delta^j)
\)
gives
\begin{align}
\|\pi\bm\delta^{j+1}\|_\rho^2
 -\|\pi\bm\delta^j\|_\rho^2
 +a(\Delta^j,\Delta^{j+1})
 -a(\Delta^{j-1},\Delta^j)
 =
 b_\rho(\bm\delta^1-\bm\delta^0,
          \bm\delta^{j+1}+\bm\delta^j)
 +\tau^2\sum_{k=1}^j
   (\bm r^k,\bm\delta^{j+1}+\bm\delta^j).
\label{eq:baker-one-step}
\end{align}
Here we used
\(
b_\rho\!\left(D_t\bm\delta^{j+\frac12},
\Delta^{j+1}-\Delta^{j-1}\right)
=
\|\pi\bm\delta^{j+1}\|_\rho^2
-\|\pi\bm\delta^j\|_\rho^2,
\)
and symmetry of \(a(\cdot,\cdot)\) for the stiffness difference.
Summing \eqref{eq:baker-one-step} once more from \(j=1\) to \(n\)
now telescopes both differences and yields the exact accumulated identity
\begin{align}
\|\pi\bm\delta^{n+1}\|_\rho^2
+a(\Delta^n,\Delta^{n+1})
=
\|\pi\bm\delta^1\|_\rho^2
+\sum_{j=1}^n
 b_\rho(\bm\delta^1-\bm\delta^0,
          \bm\delta^{j+1}+\bm\delta^j)+
\tau^2\sum_{j=1}^n\sum_{k=1}^j
(\bm r^k,\bm\delta^{j+1}+\bm\delta^j).
\label{eq:baker-telescoped}
\end{align}
In particular, no telescoping step is hidden in passing from
\eqref{eq:baker-one-step} to \eqref{eq:baker-telescoped}.
By the inverse inequality and the CFL condition (\ref{eq:cfl}),
\begin{align*}
a(\Delta^n,\Delta^{n+1})
=
\left\|\frac{\Delta^{n+1}+\Delta^n}{2}\right\|_a^2
-\frac{\tau^2}{4}\|\bm\delta^{n+1}\|_a^2
\ge
-\frac{\tau^2}{4\beta H^2}
 \|\pi\bm\delta^{n+1}\|_\rho^2.
\end{align*}
Thus
\(
c_{\rm CFL}:=1-\tau^2/(4\beta H^2)>0.
\)
Set
\(
M_n:=\max_{0\le \ell\le n+1}
\|\pi\bm\delta^\ell\|_\rho.
\)
The initial-data term in \eqref{eq:baker-telescoped} satisfies
\[
\left|
\sum_{j=1}^n
b_\rho(\bm\delta^1-\bm\delta^0,
\bm\delta^{j+1}+\bm\delta^j)
\right|
\le
2n\|\pi(\bm\delta^1-\bm\delta^0)\|_\rho M_n
\le
2T\|\pi D_t\bm\delta^{1/2}\|_\rho M_n .
\]

\paragraph{Step 2. Residual estimate.}

Denote the double residual sum in \eqref{eq:baker-telescoped} by
\(
\mathcal Q_n:=
\tau^2\sum_{j=1}^n\sum_{k=1}^j
(\bm r^k,\bm\delta^{j+1}+\bm\delta^j).
\) Accordingly, we write
$Q_n=Q_n^{\rm proj}+Q_n^{\rm unres}$,
where
\(
Q_n^{\rm proj}
:=
\tau^2\sum_{j=1}^n\sum_{k=1}^j
\bigl(\rho\pi(\rho^{-1}\bm{r}^k),
\pi(\bm{\delta}^{j+1}+\bm{\delta}^j)\bigr),
\)
and $Q_n^{\rm unres}:=Q_n-Q_n^{\rm proj}$.
We retain the factor \(\tau^2\) coming from the two discrete summations.
Using the residual decomposition from Theorem~\ref{thm:energy_error},
\[
\bm r^k
=
-\rho(I-\pi)(\rho^{-1}\bm f^k)
+\rho\pi\bigl(\partial_{tt}\bm u_h(t_k)
                 -D_{tt}\bm u_h(t_k)\bigr)
+\rho(I-\pi)\partial_{tt}\bm u_h(t_k)
+\rho\pi(D_{tt}\bm\theta^k),
\]
we split each pairing into projected and unresolved parts:
\begin{align*}
(\bm r^k,\bm v)
=
\bigl(\rho\pi(\rho^{-1}\bm r^k),\pi\bm v\bigr)
+
\bigl(\rho(I-\pi)(\rho^{-1}\bm r^k),
       (I-\pi)\bm v\bigr).
\end{align*}
For the projected part, Taylor expansion and
Lemma~\ref{lem:time_projection} give, uniformly in \(j\),
\[
\sum_{k=1}^j
\|\pi(\rho^{-1}\bm r^k)\|_\rho
\le
C\tau^{-1}(\Lambda^{-1}H^2+\tau^2).
\]
Consequently, we have
\begin{align*}
|\mathcal Q_n^{\rm proj}|
\le
2\tau^2\sum_{j=1}^n
\left(\sum_{k=1}^j
\|\pi(\rho^{-1}\bm r^k)\|_\rho\right)M_n
\le
C(\Lambda^{-1}H^2+\tau^2)M_n
\le
\eta M_n^2
+C_\eta(\Lambda^{-1}H^2+\tau^2)^2 .
\end{align*}
For the unresolved part, discrete summation by parts in the time index,
the spectral estimate \eqref{eq:spectral-estimate}, and
Theorem~\ref{thm:energy_error} give
\[
|\mathcal Q_n^{\rm unres}|
\le
C(\Lambda^{-1}H^2+\tau^2)^2,
\qquad \tau\lesssim H.
\]
This is the same unresolved-residual estimate used in the discrete Baker
argument of \cite[Theorem~4.11]{Cheung2021}; importantly, the prefactor
\(\tau^2\) in \(\mathcal Q_n\) is kept before estimating the two
time sums.  Hence
\[
|\mathcal Q_n|
\le
\eta M_n^2
+C_\eta(\Lambda^{-1}H^2+\tau^2)^2.
\]
The initial estimates already established in
Theorem~\ref{thm:energy_error} imply
\(
\|\pi\bm\delta^1\|_\rho
+\|\pi D_t\bm\delta^{1/2}\|_\rho
\le
C(\Lambda^{-1}H^2+\tau^2).
\)
Substituting these estimates into \eqref{eq:baker-telescoped}, taking the
maximum over \(n\), and choosing \(\eta<c_{\rm CFL}/2\), Young's
inequality gives
\[
\max_{0\le n\le N_T}
\|\pi\bm\delta^n\|_\rho
\le
C(\Lambda^{-1}H^2+\tau^2).
\]

\paragraph{Step 3. Estimate of $\bm{\delta}^n$.}

By 
\eqref{eq:rho-full-by-proj},
for every $\bm{v}\in \bm{V}_{\rm ms}$,
\(
\|\bm{v}\|_\rho
\le
\left(
1+\frac1{\sqrt{\Lambda\beta}}
\right)
\|\pi \bm{v}\|_\rho.
\)
Hence
\[
\max_{0\le n\le N_T}
\|\bm{\delta}^n\|_\rho
\le
C
\bigl(
\Lambda^{-1}H^2
+
\tau^2
\bigr).
\]
Consequently,
\[
\|\bar{\bm{\delta}}^{\,n+\frac12}\|_\rho
\le
C
\bigl(
\Lambda^{-1}H^2
+
\tau^2
\bigr).
\]
Finally,
\(
\|\bar{\bm{\varepsilon}}^{\,n+\frac12}\|_\rho
\le
\|\bar{\bm{\theta}}^{\,n+\frac12}\|_\rho
+
\|\bar{\bm{\delta}}^{\,n+\frac12}\|_\rho,
\)
which yields
\(
\max_{0\le n\le N_T-1}
\|\bar{\bm{\varepsilon}}^{\,n+\frac12}\|_\rho
\le
C
\bigl(
\Lambda^{-1}H^2
+
\tau^2
\bigr).
\)
This completes the proof.

\end{proof}

\begin{remark}[Density-weighted $L^2$ form]
Since
\(
\|\bm{v}\|_\rho = \|\rho^{1/2} \bm{v}\|_{L^2(\Omega)},
\)
the estimate in Theorem~\ref{estimate_for_u} is equivalently written as
\begin{equation*}
\max_{0 \le n \le N_T - 1}
\bigg\| \rho^{1/2} \frac{\bm{u}_h(t_{n+1}) - \bm{u}^{n+1}_{\mathrm{ms}} + \bm{u}_h(t_n) - \bm{u}^n_{\mathrm{ms}}}{2} \bigg\|_{L^2(\Omega)}
\le C (\Lambda^{-1} H^2 + \tau^2).
\end{equation*}
\end{remark}
Based on the estimate for the displacement, we present local stress recovery and stress error estimate as follows.
\begin{theorem} \label{estimate_for_sigma}
Assume the conditions of Theorem~\ref{thm:energy_error} and Theorem~\ref{estimate_for_u} hold.
For any $\bm{v}\in \bm{U}_h$, define the locally recovered stress and rotation
$(\mathcal S \bm{v},\mathcal G \bm{v})\in \underline{\Sigma}_h\times \Gamma_h$ by
\[
(\mathcal{A}(\mathcal S \bm{v}),\underline{w}_h)
-
\sum_{e\in\mathcal F_{dl}}
(\bm{v},\llbracket \underline{w}_h n\rrbracket)_e
+
(\operatorname{as}(\mathcal{A}\underline{w}_h),\mathcal G \bm{v})
=0,
\qquad
\forall \underline{w}_h\in \underline{\Sigma}_h,
\]
and
\[
(\operatorname{as}(\mathcal{A}(\mathcal S \bm{v})),\xi_h)=0,
\qquad
\forall \xi_h\in \Gamma_h.
\]
Let
\(
\underline{\sigma}_h^n:=\mathcal S \bm{u}_h^n,\)
\(
\underline{\sigma}_{\mathrm{ms}}^n:=\mathcal S \bm{u}_{\mathrm{ms}}^n,
\)
\(
\underline{\sigma}_H^n:=\mathcal S (P_H\bm{u}_h^n).
\)
Define the midpoint averages
\[
\bar\sigma_h^{\,n+\frac12}
=
\frac{\underline{\sigma}_h^{n+1}+\underline{\sigma}_h^n}{2},
\qquad
\bar\sigma_{\mathrm{ms}}^{\,n+\frac12}
=
\frac{\underline{\sigma}_{\mathrm{ms}}^{n+1}+\underline{\sigma}_{\mathrm{ms}}^n}{2},
\qquad
\bar\sigma_H^{\,n+\frac12}
=
\frac{\underline{\sigma}_H^{n+1}+\underline{\sigma}_H^n}{2}.
\]
Then
\[
\max_{0\le n\le N_T-1}
\left\|
\bar\sigma_H^{\,n+\frac12}
-
\bar\sigma_{\mathrm{ms}}^{\,n+\frac12}
\right\|_{\mathcal{A}}
\le
C\bigl(\Lambda^{-1}H+\tau^2\bigr).
\]
Moreover,
\[
\max_{0\le n\le N_T-1}
\left\|
\bar\sigma_h^{\,n+\frac12}
-
\bar\sigma_{\mathrm{ms}}^{\,n+\frac12}
\right\|_{\mathcal{A}}
\le
C
\bigl(
\Lambda^{-\frac12}H
+
\tau^2
\bigr).
\]
\end{theorem}

\begin{proof}
The recovery problem is local on each interaction region and is
well posed by the local solvability of the multipoint stress
control volume method. Since the recovery equations are linear in
$\bm{v}$, the operator $\mathcal S:\bm{U}_h\to\underline{\Sigma}_h$ is linear.
By the construction of the reduced bilinear form, for every
$\bm{v}\in \bm{U}_h$,
\[
a(\bm{v},\bm{v})
=
(\mathcal{A}(\mathcal{S}\bm{v}),\mathcal S \bm{v})
=
\|\mathcal S \bm{v}\|_{\mathcal{A}}^2.
\]
Equivalently,
\(
\|\mathcal S \bm{v}\|_\mathcal{A}=\|\bm{v}\|_a.
\)

We first estimate the stress recovered from the projected fine-scale
solution. By linearity of $\mathcal S$,
\[
\bar\sigma_H^{\,n+\frac12}
-
\bar\sigma_{\mathrm{ms}}^{\,n+\frac12}
=
\mathcal S
\left(
\frac{
P_H\bm{u}_h^{n+1}+P_H\bm{u}_h^n
}{2}
-
\frac{
\bm{u}_{\mathrm{ms}}^{n+1}+\bm{u}_{\mathrm{ms}}^n
}{2}
\right).
\]
Using the definition
\(
\bm{\delta}^n=\bm{u}_{\mathrm{ms}}^n-P_H\bm{u}_h^n,
\)
we get
\(
\bar\sigma_H^{\,n+\frac12}
-
\bar\sigma_{\mathrm{ms}}^{\,n+\frac12}
=
-\mathcal S\bar{\bm{\delta}}^{\,n+\frac12}.
\)
Therefore,
\(
\left\|
\bar\sigma_H^{\,n+\frac12}
-
\bar\sigma_{\mathrm{ms}}^{\,n+\frac12}
\right\|_\mathcal{A}
=
\|\bar{\bm{\delta}}^{\,n+\frac12}\|_a
\). Theorem~\ref{thm:energy_error} gives
\[
\|\bar{\bm{\delta}}^{\,n+\frac12}\|_a
\le
C\bigl(\Lambda^{-1}H+\tau^2\bigr),
\]
and hence
\[
\max_{0\le n\le N_T-1}
\left\|
\bar\sigma_H^{\,n+\frac12}
-
\bar\sigma_{\mathrm{ms}}^{\,n+\frac12}
\right\|_\mathcal{A}
\le
C\bigl(\Lambda^{-1}H+\tau^2\bigr).
\]

Next we estimate the stress recovered from the fine-grid
displacement. Recall that
\(
\bm{\varepsilon}^n
=
\bm{u}_h^n-\bm{u}_{\mathrm{ms}}^n
=
\bm{\theta}^n-\bm{\delta}^n,
\)
where
\(
\bm{\theta}^n=(I-P_H)\bm{u}_h^n.
\)
Again by linearity,
\[
\bar\sigma_h^{\,n+\frac12}
-
\bar\sigma_{\mathrm{ms}}^{\,n+\frac12}
=
\mathcal S\bar{\bm{\varepsilon}}^{\,n+\frac12}
=
\mathcal S
\left(
\bar{\bm{\theta}}^{\,n+\frac12}
-
\bar{\bm{\delta}}^{\,n+\frac12}
\right).
\]
Thus,
\[
\left\|
\bar\sigma_h^{\,n+\frac12}
-
\bar\sigma_{\mathrm{ms}}^{\,n+\frac12}
\right\|_{\mathcal{A}}
\le
\|\bar{\bm{\theta}}^{\,n+\frac12}\|_a
+
\|\bar{\bm{\delta}}^{\,n+\frac12}\|_a.
\]
By the elliptic projection estimate in Lemma~\ref{lem:elliptic_projection}, with
\(
\bm{g}(t)=\rho^{-1}\bm{f}(t)-\partial_{tt}\bm{u}_h(t),
\)
we have
\[
\|\bm{\theta}^n\|_a
=
\|(I-P_H)\bm{u}_h(t_n)\|_a
=
\|(I-P_H)G_h\bm{g}(t_n)\|_a
\le
C\Lambda^{-\frac12}H\|\bm{g}(t_n)\|_\rho.
\]
Therefore, under the regularity assumptions of Lemma~\ref{lem:elliptic_projection}, we have
\(
\|\bar{\bm{\theta}}^{\,n+\frac12}\|_a
\le
C\Lambda^{-\frac12}H.
\)
Combining this bound and the fact that $\Lambda\ge 1$ gives
\[
\max_{0\le n\le N_T-1}
\left\|
\bar\sigma_h^{\,n+\frac12}
-
\bar\sigma_{\mathrm{ms}}^{\,n+\frac12}
\right\|_A
\le
C
\bigl(
\Lambda^{-\frac12}H+\tau^2
\bigr).
\]
This completes the proof.
\end{proof}

\section{Numerical experiments}
\label{Numerical experiments}
In this section, we consider two examples to show the convergence of the proposed multiscale method. The computational domain $\Omega$ is set to be $(0,1)^2$, homogeneous
Dirichlet boundary conditions and zero initial displacement and velocity are
used in both experiments.  The density is fixed at $\rho=1$, and the elastic
parameters are determined from the P-wave speed by
\begin{equation}
  v_s=0.6v_p,
  \qquad
  \mu=\rho v_s^2,
  \qquad
  \lambda=\rho\bigl(v_p^2-2v_s^2\bigr).
  \label{eq:wave-speed-parameters}
\end{equation}
The source is a localized directional pulse centered at
$\boldsymbol{x}_0=(0.5,0.5)$:
\begin{equation}
  \boldsymbol{f}(t,\boldsymbol{x})
  = \frac{t-2/f_0}{4h^2}
    \exp\!\left[-\pi^2 f_0^2(t-2/f_0)^2\right]
    \exp\!\left[-\frac{|\boldsymbol{x}-\boldsymbol{x}_0|^2}{4h^2}\right]
    \boldsymbol{d},
  \qquad
  \boldsymbol{d}
  = \begin{pmatrix}\cos(\pi/3)\\ \sin(\pi/3)\end{pmatrix}.
  \label{eq:elastic-wave-source}
\end{equation}
We set $f_0=20$, $T=0.45$, and use a uniform $1000\times1000$ fine
partition, so that $h=10^{-3}$.  A centered second-order difference with
time step $\tau=10^{-4}$ is used for the fine-grid reference and multiscale
solutions.  The reference solution is computed with the fine-grid
multipoint stress discretization of~\cite{Fu2025}.  Twelve local spectral
functions are retained in every coarse element, and the number of
oversampling layers follows the logarithmic localization rule:
\begin{equation}
  m=\left\lceil\frac{4\log(1/H)}{\log 8}\right\rceil.
  \label{eq:oversampling-rule}
\end{equation}
The reported errors are the relative counterparts of the estimates proved in
Theorems~\ref{estimate_for_u} and~\ref{estimate_for_sigma}.  Using the midpoint averages introduced there, we set
\begin{equation}
  e_{\rho}
  =\frac{\displaystyle\max_n
    \|\overline{\boldsymbol{u}}_h^{n+1/2}
    -\overline{\boldsymbol{u}}_{\mathrm{ms}}^{n+1/2}\|_{\rho}}
  {\displaystyle\max_n
    \|\overline{\boldsymbol{u}}_h^{n+1/2}\|_{\rho}},
  \qquad
  e_{\sigma}
  =\frac{\displaystyle\max_n
    \|\overline{\boldsymbol{\sigma}}_h^{n+1/2}
    -\overline{\boldsymbol{\sigma}}_{\mathrm{ms}}^{n+1/2}\|_{A}}
  {\displaystyle\max_n
    \|\overline{\boldsymbol{\sigma}}_h^{n+1/2}\|_{A}}.
  \label{eq:numerical-relative-errors}
\end{equation}
Thus $e_{\rho}$ corresponds to the displacement estimate in Theorem~\ref{estimate_for_u},
whereas $e_{\sigma}$ corresponds to the second recovered-stress estimate in
Theorem~\ref{estimate_for_sigma}.  The observed rate between two successive coarse grids is
$\log(e_H/e_{H/2})/\log 2$.

\subsection{Binary scattering medium}

The first coefficient field is the binary scattering medium
\begin{equation}
  v_p=1+0.4\chi, \qquad \chi\in\{0,1\},
  \label{eq:binary-wave-speed}
\end{equation}
where $\chi=0$ represents the background phase and $\chi=1$ the
high-velocity phase.  Thus $v_p$ takes the values $1$ and $1.4$, with spatial
mean $1.3693712$.  Figure~\ref{fig:scatter-model} shows the medium and the
source location.

\begin{figure}[H]
  \centering
  \includegraphics[width=0.52\textwidth]{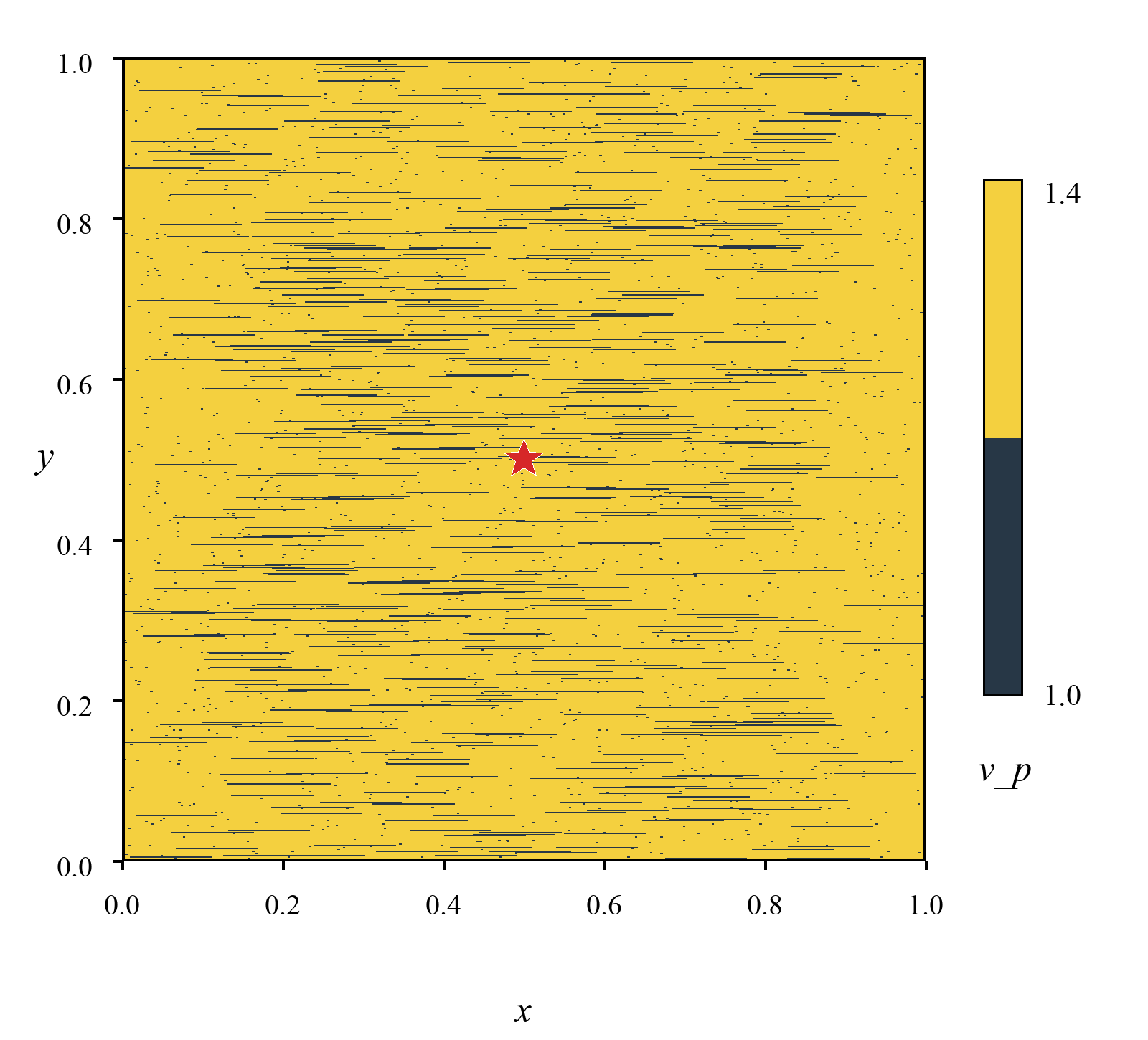}
  \caption{P-wave speed for the binary scattering medium.  The star marks the
  source location $\boldsymbol{x}_0=(0.5,0.5)$.}
  \label{fig:scatter-model}
\end{figure}

Table~\ref{tab:binary-convergence} reports the errors for four successive
coarse partitions.  Only $H$ and the corresponding oversampling size vary.

\begin{table}[H]
  \centering
  \caption{Coarse-mesh convergence for elastic-wave propagation in the
  heterogeneous scattering medium.  Twelve spectral functions are retained
  in each coarse block.}
  \label{tab:binary-convergence}
  \setlength{\tabcolsep}{9pt}
  \begin{tabular}{ccccccc}
    \toprule
    $H/h$ & $H$ & $m$ & $e_{\rho}$ & rate & $e_{\sigma}$ & rate \\
    \midrule
    40 & $1/25$ & 7 & $7.826\times10^{-1}$ & --
       & $9.507\times10^{-1}$ & -- \\
    20 & $1/50$ & 8 & $3.621\times10^{-1}$ & 1.11
       & $6.101\times10^{-1}$ & 0.64 \\
    10 & $1/100$ & 9 & $7.381\times10^{-2}$ & 2.29
       & $1.512\times10^{-1}$ & 2.01 \\
     5 & $1/200$ & 11 & $1.241\times10^{-2}$ & 2.57
       & $3.917\times10^{-2}$ & 1.95 \\
    \bottomrule
  \end{tabular}
\end{table}

Both error measures decrease consistently under coarse-grid refinement.  The
coarsest partition is pre-asymptotic, with observed rates $1.11$ and $0.64$
for displacement and recovered stress, respectively.  Once the propagating
wave field is better resolved, the rates increase substantially.  On the
final two refinements, the displacement rates are $2.29$ and $2.57$, while
the recovered-stress rates are $2.01$ and $1.95$.  At $H/h=5$, the
maximum-in-time relative errors are $1.24\%$ in the density-weighted
displacement norm and $3.92\%$ in the recovered-stress norm.  The results
demonstrate that localization with the logarithmic oversampling rule
preserves the coarse-grid convergence of the multiscale approximation for
this heterogeneous elastic-wave problem.

Figure~\ref{fig:binary-wave-components} compares the displacement components of the
reference and multiscale solutions at the final time on the finest coarse
partition.  For each component, the two panels use the same symmetric color
scale.  The multiscale solution captures the principal wave fronts together
with the fine-scale oscillations induced by the heterogeneous medium, while
retaining the phase and amplitude of the reference field.

\begin{figure}[H]
  \centering
  \includegraphics[width=0.84\textwidth]{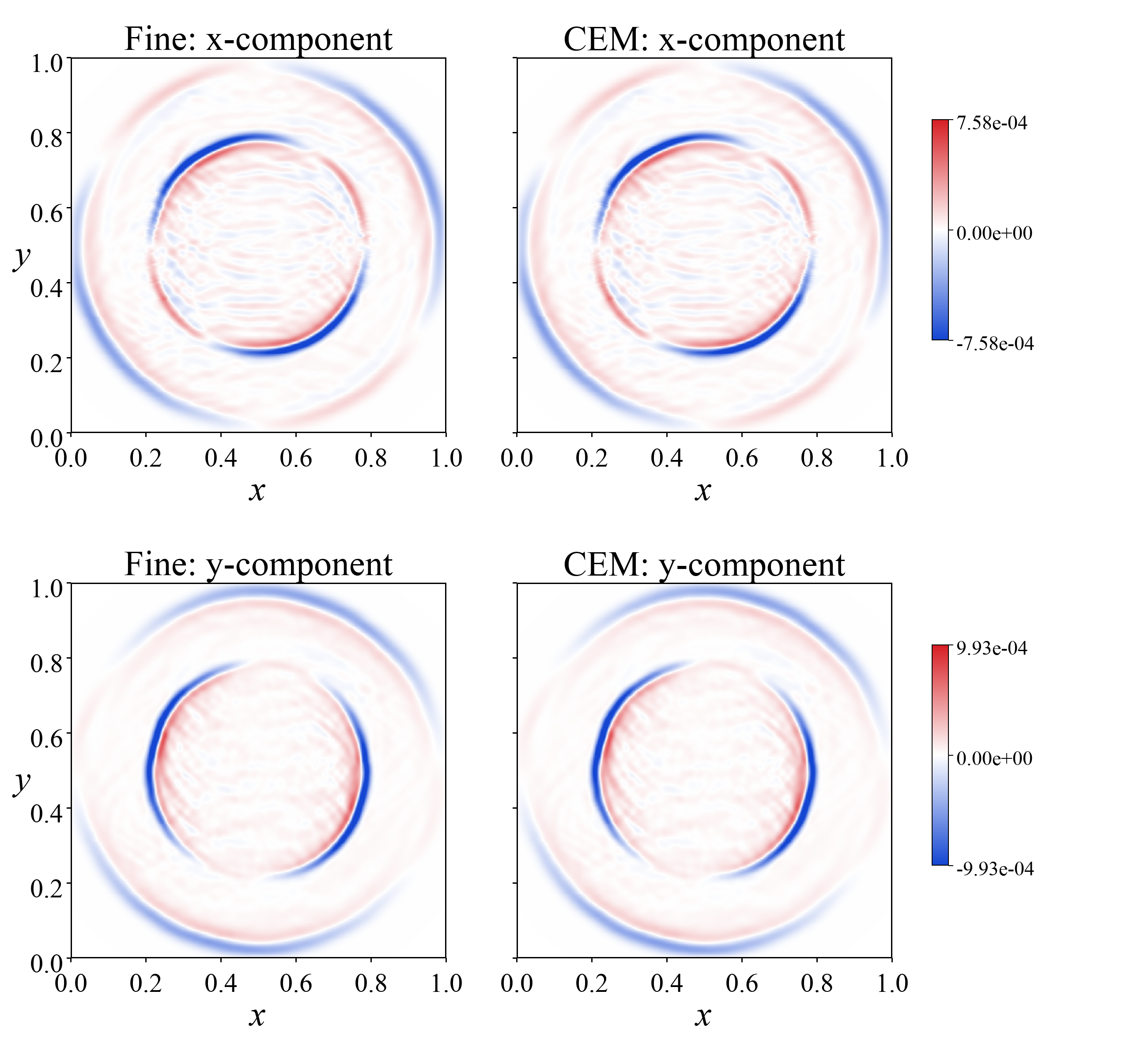}
  \caption{Displacement components at $T=0.45$.  Left: fine-grid reference
  solution.  Right: CEM solution with $H/h=5$, $m=11$, and 12 local spectral
  functions per coarse block.  Each row uses a common symmetric color scale;
  the display limit is the larger 99.5th percentile of the corresponding
  fine and CEM amplitudes.}
  \label{fig:binary-wave-components}
\end{figure}

\subsection{Strong correlated random medium}

We next assess the method in a continuous random medium with substantially
larger velocity variation.  Let $g$ be a stationary, mean-zero,
unit-variance Gaussian random field with correlation length $0.025$.  For the
realization used here, generated with a fixed seed, we define
\begin{equation}
  v_p(\boldsymbol{x})
  =0.8+\frac{1.2}{1+\exp\{-2[g(\boldsymbol{x})-0.110557]\}}.
  \label{eq:random-wave-speed}
\end{equation}
The shift in \eqref{eq:random-wave-speed} is selected so that the mean
P-wave speed remains $1.3693712$, matching the binary experiment.  The
realized field satisfies
\[
  0.8009\le v_p\le1.9968,
  \qquad \operatorname{std}(v_p)=0.3787,
\]
and has a velocity ratio of approximately $2.49$.  Figure~\ref{fig:random-model}
shows the coefficient field.  In contrast to the binary medium, the
heterogeneity fills the entire domain and contains connected slow and fast
regions over a range of resolved spatial scales.

\begin{figure}[H]
  \centering
  \includegraphics[width=0.52\textwidth]{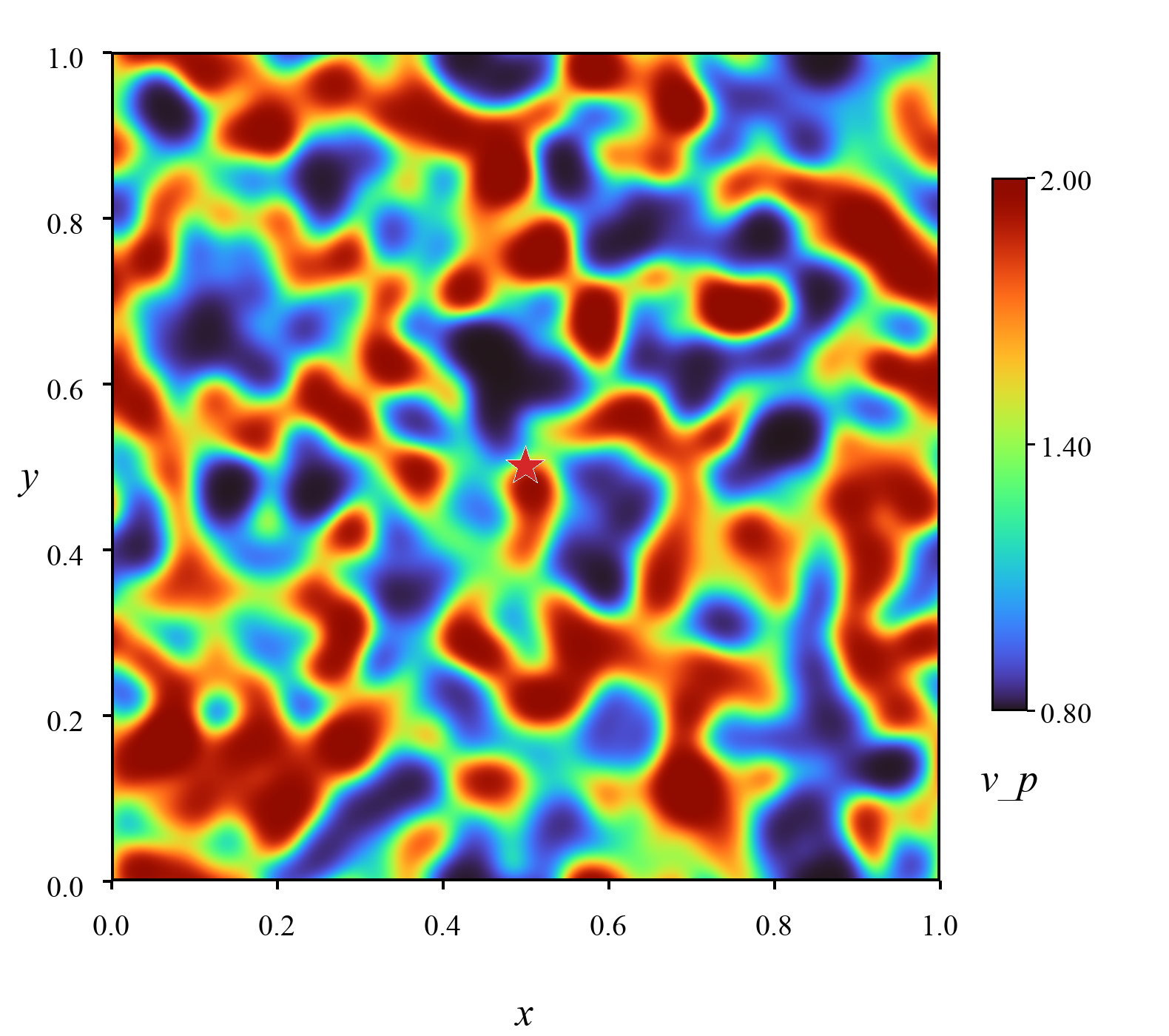}
  \caption{P-wave speed for the strong correlated random medium.  The star
  marks the source location $\boldsymbol{x}_0=(0.5,0.5)$.}
  \label{fig:random-model}
\end{figure}

All discretization and source parameters are unchanged.  Table~\ref{tab:random-convergence}
shows that the stronger coefficient variation delays the asymptotic regime:
the two coarsest spaces do not yet resolve the multiple scattering produced
by the random medium.  Nevertheless, refinement from $H/h=20$ to $10$
reduces the displacement and recovered-stress errors at rates $2.03$ and
$1.74$, respectively.  On the final refinement both observed rates are
approximately $2.75$.

\begin{table}[H]
  \centering
  \caption{Coarse-mesh convergence for the strong correlated random medium.
  Twelve spectral functions are retained in every coarse block.}
  \label{tab:random-convergence}
  \setlength{\tabcolsep}{9pt}
  \begin{tabular}{ccccccc}
    \toprule
    $H/h$ & $H$ & $m$ & $e_{\rho}$ & rate & $e_{\sigma}$ & rate \\
    \midrule
    40 & $1/25$  & 7  & $9.271\times10^{-1}$ & --
       & $1.084\times10^{0}$  & -- \\
    20 & $1/50$  & 8  & $5.612\times10^{-1}$ & 0.72
       & $8.889\times10^{-1}$ & 0.29 \\
    10 & $1/100$ & 9  & $1.371\times10^{-1}$ & 2.03
       & $2.668\times10^{-1}$ & 1.74 \\
     5 & $1/200$ & 11 & $2.041\times10^{-2}$ & 2.75
       & $3.943\times10^{-2}$ & 2.76 \\
    \bottomrule
  \end{tabular}
\end{table}

At $H/h=5$, the maximum-in-time errors are $2.04\%$ in the
density-weighted displacement norm and $3.94\%$ in the recovered-stress
norm.  No error plateau is observed over the tested range of coarse meshes.

Figure~\ref{fig:random-wave-components} compares the fine and multiscale
solutions on the finest coarse partition.  The multiscale field closely
reproduces the strongly distorted wave fronts, localized focusing, and
multiple-scattering patterns of the reference solution.

\begin{figure}[H]
  \centering
  \includegraphics[width=0.74\textwidth]{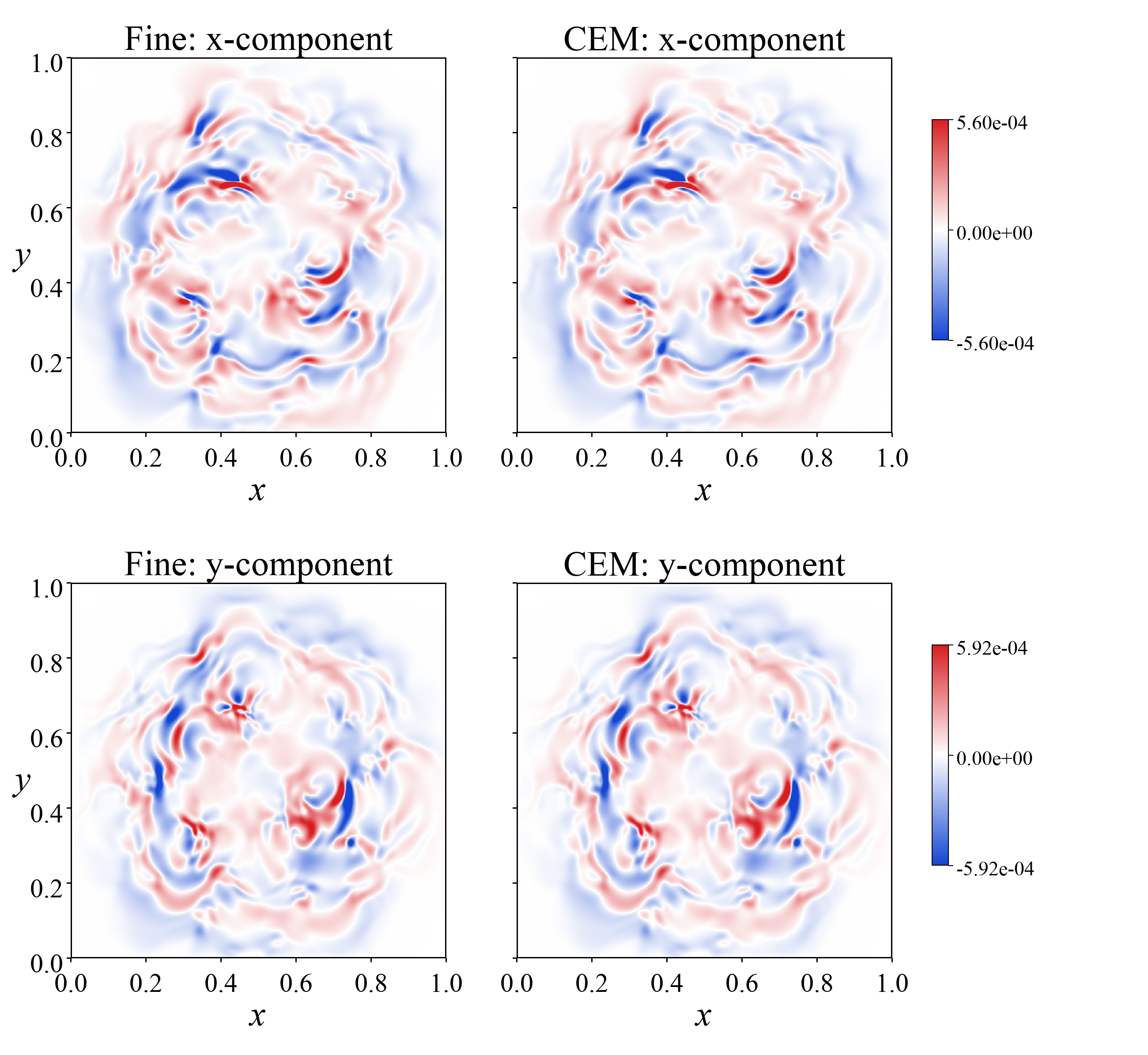}
  \caption{Strong random medium: displacement components at $T=0.45$.
  Left: fine-grid reference solution.  Right: CEM solution with $H/h=5$,
  $m=11$, and 12 local spectral functions per coarse block.  Each row uses a
  common symmetric color scale based on the larger 99.5th percentile of the
  fine and CEM amplitudes.}
  \label{fig:random-wave-components}
\end{figure}

\section{Conclusions}
\label{conclusions}
In this work, we developed a time explicit multiscale framework for
elastic wave propagation in heterogeneous media by combining local multipoint
stress condensation with a density-weighted localized multiscale reduction. The
resulting formulation evolves only the cell-centered displacement globally,
retains inexpensive local recovery of the fine-scale stress, and produces an
identity coarse mass matrix that enables explicit time stepping without global
linear solves. Theoretical stability and convergence results, together with
numerical experiments in binary scattering and strongly correlated random media,
confirm the accuracy and effectiveness of the proposed method.


\section*{Declaration of competing interest}

The authors declare that they have no known competing financial interests or personal relationships that could have appeared
to influence the work reported in this paper.

\section*{Declaration of Generative AI and AI-assisted technologies in the writing process}

During the preparation of this work the authors used ChatGPT in order to improve readability and language. After
using this tool, the authors reviewed and edited the content as needed and take full responsibility for the content of the
publication.

\section*{Acknowledgments}
Eric Chung’s research is partially supported by the Hong Kong RGC General Research Fund Projects 14304525 and 14305624. 

\bibliographystyle{siamplain}

\end{document}